\documentclass[11pt,a4paper]{article}
\usepackage[margin=25mm]{geometry}
\usepackage[T1]{fontenc}
\usepackage{lmodern}
\usepackage{needspace,amsmath,amssymb,mathtools,amsthm}
\usepackage{enumitem,graphicx,float,tikz,etoolbox}
\usepackage[numbers,sort&compress]{natbib}
\usepackage[hidelinks,bookmarksdepth=2]{hyperref}
\usepackage[nameinlink,noabbrev]{cleveref}
\usetikzlibrary{arrows.meta,calc,positioning}
\providecommand{\bibcommenthead}{}
\providecommand{\doi}[1]{\href{https://doi.org/#1}{\nolinkurl{https://doi.org/#1}}}
\theoremstyle{plain}
\newtheorem{theorem}{Theorem}[section]
\newtheorem{lemma}[theorem]{Lemma}
\newtheorem{proposition}[theorem]{Proposition}

\theoremstyle{definition}
\newtheorem{definition}[theorem]{Definition}
\theoremstyle{remark}
\newtheorem{remark}[theorem]{Remark}
\newtheorem*{remark*}{Remark}
\makeatletter
\renewenvironment{proof}[1][\proofname]{%
  \par\pushQED{\qed}\normalfont\topsep6\p@\@plus6\p@\relax
  \trivlist\item[\hskip\labelsep\itshape
  \ifstrempty{#1}{}{#1\@addpunct{.}}]\ignorespaces
}{\popQED\endtrivlist\@endpefalse}
\makeatother
\newcommand{\R}{\mathbb R}

\newcommand{\D}{\mathcal D}
\newcommand{\diver}{\operatorname{div}}
\newcommand{\dashint}{\mathop{\rlap{\raisebox{0.15ex}{--}}\!\int}}
\newcommand{\fint}{\dashint}
\AtBeginDocument{\allowdisplaybreaks[0]}
\AtBeginEnvironment{thebibliography}{\interlinepenalty=10000}
\numberwithin{equation}{section}
\numberwithin{figure}{section}
\numberwithin{table}{section}
\newcommand{\ManuscriptTitle}{Quantitative Homogenization and Uniform Lipschitz Estimates for Lam\'e--Stokes Systems}

\newcommand{\ManuscriptAbstract}{We study homogenization of Lam\'e--Stokes systems with fixed material parameters.
Incompressible elastic inclusions are periodically distributed in a connected elastic matrix,
with an elastic buffer near the outer boundary.
The divergence constraint acts only in the inclusions,
and the full pressure contributes to interface traction. We identify the
effective elasticity system for the traction problem of period
\(\varepsilon\). Under the respective regularity assumptions,
we obtain the optimal \(O(\varepsilon)\) convergence rate for the displacement in \(L^2\).
For the unsmoothed first-order approximation, the displacement error in \(H^1\)
and the pressure error in \(L^2\) over the inclusions are both \(O(\varepsilon^{1/2})\). For a \(C^{1,1}\) outer boundary, both errors improve to \(O(\varepsilon)\)
on every fixed interior subdomain. We establish local interface Schauder estimates
including the full pressure and combine them with large-scale iteration to obtain
Lipschitz estimates for the displacement up
to the boundary and \(L^\infty\) bounds for the pressure in the inclusions.
These Lipschitz and pressure bounds are uniform in \(\varepsilon\) for \(C^{1,\kappa}\) interfaces
and outer boundary and compatible \(C^{0,\tau}\) traction data, where
\(0<\tau<\kappa<1\).}
\newcommand{\ManuscriptKeywords}{Lam\'e--Stokes systems, quantitative homogenization, uniform Lipschitz estimates, incompressible inclusions}

\title{\ManuscriptTitle}
\author{Beichen Wang\\[0.4em]
  \small Qiuzhen College, Tsinghua University\\
  \small Beijing, China\\
  \small\href{mailto:wbc23@mails.tsinghua.edu.cn}{wbc23@mails.tsinghua.edu.cn}}
\date{}
\begin{document}
\hypersetup{pdftitle={\ManuscriptTitle},pdfauthor={Beichen Wang},pdfsubject={Periodic homogenization of Lame--Stokes composites}}
\maketitle
\begin{abstract}
\ManuscriptAbstract
\end{abstract}
\section{Introduction}
\label{sec:unified-introduction}

We consider periodically distributed incompressible elastic inclusions in
a compressible elastic matrix. The inclusions retain finite shear stiffness
and may change shape, while their first-order volume change vanishes.
Fu and Jing \cite{FuJing2024} derive the Lam\'e--Stokes system describing
this static equilibrium as the inclusion bulk modulus tends to infinity
with fixed shear modulus; the vector unknown in the Stokes phase is an elastic
displacement. With the material parameters fixed, we ask how accurately an
effective elastic material approximates the overall deformation as the
microstructure is refined, and whether local strain and inclusion pressure
remain uniformly controlled up to the exterior boundary.

\label{subsec:unified-microscopic-models}\label{sec:unified-model}\label{subsec:unified-mixed-main}

\noindent\textbf{(G)} Let \(d\ge2\), and let \(\Omega\subset\R^d\) be a
bounded connected Lipschitz domain. Set \(Y=(-\tfrac12,\tfrac12)^d\).
The reference inclusion \(\omega\Subset Y\) is a nonempty connected
Lipschitz open set, and \(Y\setminus\overline\omega\) is connected.
For \(k\in\mathbb Z^d\) and \(0<\varepsilon<1\), set
\[
 Y_\varepsilon^k=\varepsilon(k+Y),\qquad
 \omega_\varepsilon^k=\varepsilon(k+\omega),\qquad
 \Pi_\varepsilon=\{k\in\mathbb Z^d:\overline{Y_\varepsilon^k}\subset\Omega\}.
\]
We normalize the macroscopic length so that \(\Pi_\varepsilon\ne\varnothing\)
for \(0<\varepsilon<1\). The inclusion and matrix regions are
\begin{equation}\label{eq:complete-cell-geometry}
 D_\varepsilon=\bigcup_{k\in \Pi_\varepsilon}\omega_\varepsilon^k,
 \qquad \Omega_\varepsilon=\Omega\setminus\overline{D_\varepsilon}.
\end{equation}
Only complete cells carry inclusions. The remaining boundary buffer is
filled with the same elastic material as the matrix. The distance of every
inclusion from \(\partial\Omega\) is at least
\(\varepsilon\operatorname{dist}(\overline\omega,\partial Y)\).
Figure~\ref{fig:positive-geometry} shows this geometry.

\begin{figure}[htbp]
\centering
\def\GeometryElasticLegend{Lam\'e phase\newline (both green shades)}
\def\GeometryStokesLegend{Stokes phase\newline and pressure support}
\def\GeometryCellLegend{Cell edges\newline (not physical interfaces)}
\def\GeometryBufferLabel{Buffer}

\resizebox{.90\linewidth}{!}{
\definecolor{cmpcellgreen}{RGB}{48,174,96}
\definecolor{cmpinclusionyellow}{RGB}{255,190,35}
\definecolor{cmpdomainred}{RGB}{190,46,51}
\begin{tikzpicture}[
  x=.54cm,y=.54cm,>=Latex,line cap=round,line join=round,
  every node/.style={font=\small},
  cell grid/.style={draw=black!48,dash pattern=on 1.7pt off 1.5pt,line width=.35pt},
  geometry pointer/.style={->,draw=black!80,line width=.55pt}
]
  \def\PositiveBodyPath{
    (7.00,12.82)
    .. controls (7.71,12.82) and (11.58,12.24) .. (12.20,12.00)
    .. controls (12.82,11.76) and (14.40,10.50) .. (14.48,10.00)
    .. controls (14.56,9.50) and (13.41,6.50) .. (13.20,6.00)
    .. controls (12.99,5.50) and (12.37,4.50) .. (12.00,4.00)
    .. controls (11.63,3.50) and (9.19,.43) .. (8.75,0.00)
    .. controls (8.31,-.43) and (7.12,-1.18) .. (6.70,-1.18)
    .. controls (6.28,-1.18) and (4.15,-.26) .. (3.75,0.00)
    .. controls (3.35,.26) and (2.24,1.46) .. (1.90,2.00)
    .. controls (1.56,2.54) and (-.33,5.75) .. (-.30,6.50)
    .. controls (-.27,7.25) and (1.97,10.54) .. (2.30,11.00)
    .. controls (2.63,11.46) and (3.31,11.85) .. (3.70,12.00)
    .. controls (4.09,12.15) and (6.29,12.82) .. cycle
  }
  \tikzset{positive cell/.pic={
    \path[fill=cmpcellgreen!48] (0,0) rectangle (1,1);
    \path[fill=cmpinclusionyellow,draw=orange!70!black,line width=.30pt]
      (.16,.55)
      .. controls (.12,.65) and (.27,.70) .. (.36,.76)
      .. controls (.48,.87) and (.65,.80) .. (.78,.68)
      .. controls (.84,.58) and (.79,.43) .. (.70,.38)
      .. controls (.57,.34) and (.47,.23) .. (.32,.30)
      .. controls (.20,.36) and (.25,.48) .. cycle;
  }}

  \path[fill=cmpcellgreen!14] \PositiveBodyPath;
  \foreach \x in {4,...,7}  \pic at (\x,0)  {positive cell};
  \foreach \x in {3,...,8}  \pic at (\x,1)  {positive cell};
  \foreach \x in {2,...,9}  \pic at (\x,2)  {positive cell};
  \foreach \x in {2,...,10} \pic at (\x,3)  {positive cell};
  \foreach \x in {1,...,10} \pic at (\x,4)  {positive cell};
  \foreach \x in {1,...,11} \pic at (\x,5)  {positive cell};
  \foreach \x in {0,...,12} \pic at (\x,6)  {positive cell};
  \foreach \x in {1,...,12} \pic at (\x,7)  {positive cell};
  \foreach \x in {1,...,12} \pic at (\x,8)  {positive cell};
  \foreach \x in {2,...,13} \pic at (\x,9)  {positive cell};
  \foreach \x in {3,...,12} \pic at (\x,10) {positive cell};
  \foreach \x in {4,...,11} \pic at (\x,11) {positive cell};

  \begin{scope}
    \clip \PositiveBodyPath;
    \foreach \x in {0,...,15}
      \draw[cell grid] (\x,-1.5) -- (\x,13.5);
    \foreach \y in {-1,...,13}
      \draw[cell grid] (-1,\y) -- (15.5,\y);
  \end{scope}
  \path[draw=cmpdomainred,line width=.95pt] \PositiveBodyPath;
  \node at (6.5,13.55) {$\Omega$};

  \node[anchor=west] at (16.05,13.0) {$\partial\Omega$};
  \draw[geometry pointer] (15.85,12.92) -- (12.20,12.00);
  \node[anchor=west] at (16.05,11.35) {\GeometryBufferLabel};
  \draw[geometry pointer] (15.85,11.35) -- (12.55,11.30);
  \node[anchor=west] at (16.05,5.4) {$\Omega_\varepsilon$};
  \draw[geometry pointer] (15.85,5.35) -- (11.90,5.07);
  \node[anchor=west] at (16.05,3.6) {$D_\varepsilon$};
  \draw[geometry pointer] (15.85,3.60) -- (10.48,3.60);

  \begin{scope}[shift={(20,8)},scale=3.4]
    \path[fill=cmpcellgreen!48] (0,0) rectangle (1,1);
    \path[fill=cmpinclusionyellow,draw=orange!70!black,line width=.30pt]
      (.16,.55)
      .. controls (.12,.65) and (.27,.70) .. (.36,.76)
      .. controls (.48,.87) and (.65,.80) .. (.78,.68)
      .. controls (.84,.58) and (.79,.43) .. (.70,.38)
      .. controls (.57,.34) and (.47,.23) .. (.32,.30)
      .. controls (.20,.36) and (.25,.48) .. cycle;
    \draw[cell grid] (0,0) rectangle (1,1);
  \end{scope}
  \node at (21.7,12.0) {$Y_\varepsilon^k$};
  \node at (21.7,9.85) {$\omega_\varepsilon^k$};
  \draw[<->,line width=.45pt] (20,7.4) -- (23.4,7.4);
  \draw[line width=.4pt] (20,7.2) -- (20,7.65);
  \draw[line width=.4pt] (23.4,7.2) -- (23.4,7.65);
  \node[fill=white,inner sep=2pt] at (21.7,7.4) {$\varepsilon$};

  \path[fill=cmpcellgreen!48] (18.8,4.60) rectangle (19.2,5.20);
  \path[fill=cmpcellgreen!14] (19.2,4.60) rectangle (19.6,5.20);
  \draw[black!35,line width=.25pt] (18.8,4.60) rectangle (19.6,5.20);
  \node[anchor=west,align=left,text width=4.15cm] at (19.95,4.90)
    {\GeometryElasticLegend};
  \path[fill=cmpinclusionyellow,draw=orange!70!black,line width=.3pt]
    (18.8,2.50) rectangle (19.6,3.10);
  \node[anchor=west,align=left,text width=4.15cm] at (19.95,2.80)
    {\GeometryStokesLegend};
  \draw[cell grid] (18.8,.75) -- (19.6,.75);
  \node[anchor=west,align=left,text width=4.15cm] at (19.95,.75)
    {\GeometryCellLegend};
\end{tikzpicture}}
\caption{Periodic Stokes inclusions in a connected Lam\'e matrix.
The gold inclusions form \(D_\varepsilon\), where the pressure is supported.
Both green shades represent the same elastic material, with the lighter
shade marking the boundary buffer. The red curve is \(\partial\Omega\);
dashed cell edges are not material interfaces. The enlarged cell has side
length \(\varepsilon\).}

\label{fig:positive-geometry}
\end{figure}

The material parameters of the periodic problem are constant in each phase,
independent of \(\varepsilon\), and satisfy
\begin{equation}\label{eq:material-conditions}
 \mu>0,\qquad d\lambda+2\mu>0,\qquad \widetilde\mu>0.
\end{equation}
Component indices range from \(1\) to \(d\), and repeated component indices are summed
unless they label a fixed cell problem. We write
\[
 (\nabla v)_i^\alpha=\partial_{x_\alpha}v^i,\qquad
 \D v=\tfrac12\bigl(\nabla v+(\nabla v)^T\bigr),\qquad
 \diver v=\operatorname{tr}(\nabla v).
\]
The action of a fourth-order tensor on a matrix is defined by \((AM)_i^\alpha=A_{ij}^{\alpha\beta}M_j^\beta\).
The Lam\'e and Stokes stresses are
\begin{equation}\label{eq:unified-stresses}
 \sigma(v)=\lambda\diver v\,\mathbb I+2\mu\D v,\qquad
 \sigma(v,q)=2\widetilde\mu\D v+q\mathbb I.
\end{equation}

Here \(\mathbb I\) is the identity matrix, and matrix divergence is taken
row by row. Let \(\mathcal R\) denote the rigid displacements.
Using their boundary traces, define
\[
\begin{aligned}
 H_{\mathcal R}^{1/2}(\partial\Omega)
 &=\left\{v\in H^{1/2}(\partial\Omega;\R^d):
 \int_{\partial\Omega}v\cdot r\,\mathrm dS=0
 \quad\forall r\in\mathcal R\right\},\\
 H_{\mathcal R}^{-1/2}(\partial\Omega)
 &=\left\{g\in H^{-1/2}(\partial\Omega;\R^d):
 \langle g,r\rangle_{\partial\Omega}=0
 \quad\forall r\in\mathcal R\right\}.
\end{aligned}
\]
For the periodic problem, we seek a displacement
\(u_\varepsilon\in H^1(\Omega;\R^d)\) and a pressure
\(p_\varepsilon\in L^2(D_\varepsilon)\) satisfying
\begin{equation}\label{eq:strong-stokes-inclusions}
 \begin{cases}
  \diver\sigma(u_\varepsilon)=0
      &x\in\Omega_\varepsilon,\\
  \diver\sigma(u_\varepsilon,p_\varepsilon)=0,
  \quad\diver u_\varepsilon=0
      &x\in D_\varepsilon,\\
  u_\varepsilon|_-=u_\varepsilon|_+,
  \quad\sigma(u_\varepsilon,p_\varepsilon)\nu|_-
       =\sigma(u_\varepsilon)\nu|_+
      &x\in\partial D_\varepsilon,\\
  \multicolumn{2}{l}{\sigma(u_\varepsilon)n=g\in H_{\mathcal R}^{-1/2}(\partial\Omega),
  \quad u_\varepsilon|_{\partial\Omega}\in H_{\mathcal R}^{1/2}(\partial\Omega).}
 \end{cases}
\end{equation}

The normal \(n\) points out of \(\Omega\), and \(\nu\) points from
\(D_\varepsilon\) into \(\Omega_\varepsilon\); the signs \(-\) and \(+\)
denote the corresponding traces. The traction conditions are understood
weakly. The orthogonality condition on \(u_\varepsilon\) fixes the
rigid-motion freedom.

Pressure acts as the reaction to the incompressibility constraint and
contributes to interface traction. Changing a pressure constant changes
the normal traction, so these constants are fixed by the coupling between
the phases and a mean cannot be subtracted arbitrarily.
Section~\ref{sec:unified-stability} gives
the mixed variational formulation and proves well-posedness.

\subsection{Main results}
\label{subsec:main-results}

Fu and Jing \cite{FuJing2025} establish convergence rates and Lipschitz
estimates uniform in both period and contrast for Dirichlet problems
for high-contrast elliptic systems. The traction problem considered here
also imposes incompressibility in the inclusions and requires joint
estimates for displacement and full pressure.

\noindent\begin{minipage}{\linewidth}
\begin{theorem}[Quantitative homogenization]
\label{thm:main-quantitative}
Assume \textup{(G)}, and let \((u_\varepsilon,p_\varepsilon)\), \(u_0\)
and \((\chi_j^\beta,r_j^\beta)\) be determined by
\eqref{eq:strong-stokes-inclusions}, \eqref{eq:unified-effective-problem}
and \eqref{eq:unified-cell-momentum}, respectively.
For \(0<\varepsilon<1\), we have:
\begin{enumerate}[label=\textup{(\roman*)},leftmargin=*]
\item \textbf{Displacement error.} If \(\partial\Omega\in C^{1,1}\) and
\(u_0\in H^2(\Omega;\R^d)\), then
\begin{equation}
 \|u_\varepsilon-u_0\|_{L^2(\Omega)}
 \le C\varepsilon\|u_0\|_{H^2(\Omega)}.
 \label{eq:main-L2-rate}
\end{equation}
\item \textbf{Joint error of the first-order approximation.}
There exists an exponent \(q\),
depending only on \(d,\lambda,\mu,\widetilde\mu,\omega\),
with \(q=2\) if \(d=2\) and \(2<q<d\) if \(d\ge3\),
such that the following estimates hold when \(u_0\in W^{2,q}(\Omega;\R^d)\).

\noindent\textbf{Global estimate:}
\begin{equation}\label{eq:main-raw-H2-joint-rate}
 \begin{aligned}
 &\bigl\|u_\varepsilon-u_0
       -\varepsilon\chi_j^\beta(x/\varepsilon)(\D u_0)_j^\beta\bigr\|_{H^1(\Omega)}\\
 &\quad+\bigl\|p_\varepsilon
       -r_j^\beta(x/\varepsilon)(\D u_0)_j^\beta\bigr\|_{L^2(D_\varepsilon)}
 \le C\varepsilon^{1/2}\|u_0\|_{W^{2,q}(\Omega)}.
 \end{aligned}
\end{equation}
\noindent\textbf{Interior estimate:} If \(\partial\Omega\in C^{1,1}\), then for every
\(\Omega'\Subset\Omega\),
\begin{equation}\label{eq:main-interior-joint-rate}
 \begin{aligned}
 &\bigl\|u_\varepsilon-u_0
       -\varepsilon\chi_j^\beta(x/\varepsilon)(\D u_0)_j^\beta\bigr\|_{H^1(\Omega')}\\
 &\quad+\bigl\|p_\varepsilon
       -r_j^\beta(x/\varepsilon)(\D u_0)_j^\beta\bigr\|_{L^2(D_\varepsilon\cap\Omega')}
 \le C_{\Omega'}\varepsilon\|u_0\|_{W^{2,q}(\Omega)}.
 \end{aligned}
\end{equation}
If \(\partial\omega\in C^{1,\kappa}\) with \(0<\kappa<1\),
then it suffices to assume \(u_0\in H^2(\Omega;\R^d)\):
estimates \eqref{eq:main-raw-H2-joint-rate} and \eqref{eq:main-interior-joint-rate}
hold with the \(H^2(\Omega)\) norm on the right-hand sides;
the interior estimate still requires \(\partial\Omega\in C^{1,1}\).
\end{enumerate}

The constants depend only on the dimension, material parameters,
the corresponding geometric data and the relevant exponents, and are independent
of \(\varepsilon\) and \(g\); the interior constant also depends on
\(\operatorname{dist}(\Omega',\partial\Omega)\).
\end{theorem}
\end{minipage}

The convergence order in \eqref{eq:main-L2-rate} is optimal in general.
Indeed, the mismatch between the inclusion incompressibility constraint
and a nonzero volumetric strain of the effective displacement can produce
a matching lower bound; see Proposition~\ref{prop:unified-L2-sharpness}.

Approximation error in an averaged sense does not directly control local
strain and pressure. We therefore also seek uniform pointwise estimates
as the microstructure is refined.

\begin{samepage}
\begin{theorem}[Global Lipschitz and pressure estimates]
\label{thm:positive-boundary-pressure}
Assume \textup{(G)} and
\(\partial\Omega,\partial\omega\in C^{1,\kappa}\), where
\(0<\tau<\kappa<1\).
For \(g\in C^{0,\tau}(\partial\Omega;\R^d)\cap H_{\mathcal R}^{-1/2}(\partial\Omega)\),
let \((u_\varepsilon,p_\varepsilon)\) solve
\eqref{eq:strong-stokes-inclusions}.
Then, for every \(0<\varepsilon<1\),
\begin{equation}\label{pos:global-result}
 \|\nabla u_\varepsilon\|_{L^\infty(\Omega)}
 +\|p_\varepsilon\|_{L^\infty(D_\varepsilon)}
 \le C\|g\|_{C^{0,\tau}(\partial\Omega)}.
\end{equation}
The constant \(C\) depends only on \(d,\kappa,\tau,\lambda,\mu,\widetilde\mu,\Omega,\omega\) and is independent of \(\varepsilon\) and \(g\).
\end{theorem}
\end{samepage}

For periodic Stokes systems, Gu and Shen \cite{GuShen2015} establish
interior Lipschitz estimates for velocity and uniform estimates for pressure
oscillations, while Gu and Xu \cite{GuXu2017Boundary} prove boundary
Lipschitz and pressure estimates for Dirichlet problems. These pressure
estimates are stated after subtracting a mean. In the present coupled
problem, the pressure constants are determined by the interface traction
and enter the transmission of stress between the two phases.
Theorem~\ref{thm:positive-boundary-pressure} controls both the pressure
oscillations and the constant part in each inclusion, with a bound
independent of the period.

\subsection{Related work and proof strategy}
\label{subsec:related-work}
\label{subsec:proof-strategy}

Mercier and Nicaise \cite{MercierNicaise2012} study Lam\'e--Stokes
interface problems and the incompressible limit in two-dimensional nonsmooth
domains using a mixed formulation with the full pressure space.
For our traction problem, divergence liftings on the fixed outer domain
yield mixed stability uniform in the period.

The piecewise estimates of Li and Vogelius \cite{LiVogelius2000} for scalar
equations were extended to systems by Li and Nirenberg
\cite{LiNirenberg2003}, who also obtain uniform interior gradient bounds
for periodic composites. Zhuge \cite{Zhuge2021} uses piecewise linear
solutions and Campanato iteration for transmission regularity and
homogenization. For Stokes systems with piecewise coefficients and a
divergence equation throughout the domain,
\cite{ChoiDongXu2022,DongLiXu2024} give gradient--pressure
bounds and, under stronger smoothness assumptions, higher regularity.

Theorem~\ref{thm:main-interface-schauder} establishes local interface Schauder estimates
for the Lam\'e--Stokes transmission problem without requiring periodicity.
For \(C^{k,\kappa}\) interfaces and correspondingly regular data, it gives
piecewise \(C^{k,\kappa}\) estimates for the displacement and
\(C^{k-1,\kappa}\) estimates for the full pressure in the inclusion phase,
in every dimension \(d\ge2\), for integers \(k\ge1\) and \(0<\kappa<1\).
Here the divergence constraint acts only in the inclusions.
We follow the above framework of piecewise approximation and Campanato iteration.
Local divergence liftings and joint displacement--pressure estimates control
the coefficient perturbation terms involving the full pressure and the
divergence defect in the frozen comparison after flattening the interface.
The case \(k=1\), together with the pure-phase estimates, controls the
cell correctors and solutions below the period.

For smoothed corrector estimates, see \cite{ZhikovPastukhova2005}.
Kenig, Lin and Shen \cite{KenigLinShen2012} establish \(L^2\) rates for elliptic systems in Lipschitz domains.
Shen and Zhuge \cite{ShenZhuge2017Elasticity} combine smoothing, boundary-layer
estimates and duality for bounded measurable elasticity tensors.
For Stokes systems, Gu \cite{Gu2016,Gu2018Neumann} obtains joint velocity--pressure
estimates for Dirichlet and Neumann problems; Xu \cite{Xu2017} extends Dirichlet rates to Lipschitz domains.

The smoothed, boundary-truncated first-order approximation generally
fails to satisfy inclusion incompressibility exactly.
We combine flux correctors \cite{Shen2018,ShenZhuge2017Elasticity} with mixed stability
to control the momentum residual, inclusion divergence defect and elastic buffer
contribution, obtaining the global joint error bound.
Duality and local energy estimates give the \(L^2\) displacement and interior joint bounds.
Under the corresponding assumptions of Theorem~\ref{thm:main-quantitative},
corrector boundedness for smooth interfaces or higher integrability for Lipschitz interfaces removes smoothing.

Uniform Lipschitz estimates build on the compactness method of Avellaneda
and Lin \cite{AvellanedaLin1987}. Kenig, Lin and Shen
\cite{KenigLinShen2013Neumann} treat Neumann problems for symmetric
periodic elliptic systems, while Gu and Shen \cite{GuShen2015} use
compactness for Stokes systems. For our interior Lipschitz estimate, local energy bounds give compactness. Reciprocal tests with the original solution and corrected affine solutions identify the effective equation through the complete stress; corrected affine approximation then transfers its regularity down to the period scale.

Armstrong and Smart \cite{ArmstrongSmart2016} formulate a large-scale iteration
from quantitative approximation and improvement of flatness;
Shen \cite{Shen2017Boundary} develops this approach for periodic elasticity with traction boundaries.
We obtain large-scale boundary estimates by comparison with an effective solution
having the same traction, using mixed stability to control the inclusion divergence defect.
Combining these with interior estimates below the period and pure Lam\'e estimates
in the boundary buffer gives the global gradient bound.
Local mixed estimates and pressure harmonicity in the inclusions then give the pointwise full pressure bound.

Sections~\ref{sec:unified-cell} and \ref{sec:unified-interface-regularity}
provide the mixed, cell and local regularity theory;
Sections~\ref{sec:unified-quantitative-architecture}--\ref{sec:uniform-lipschitz-pressure}
prove convergence rates and interior and boundary uniform estimates.

\section{Mixed formulation and effective elasticity}
\label{sec:unified-cell}

For fixed material parameters and the Lipschitz geometry \textup{(G)},
we establish well-posedness of \eqref{eq:strong-stokes-inclusions}
and bounds uniform in \(\varepsilon\). We define the effective elasticity tensor through the cell
problems and give the effective traction problem \eqref{eq:unified-effective-problem}.

\subsection{Stability and well-posedness of the mixed problem}
\label{sec:unified-stability}

To write \eqref{eq:strong-stokes-inclusions} in mixed variational form, define
the displacement space
\[
 V=\{v\in H^1(\Omega;\R^d):
     v|_{\partial\Omega}\in H_{\mathcal R}^{1/2}(\partial\Omega)\},
 \qquad \|v\|_V:=\|v\|_{H^1(\Omega)}.
\]
The pressure space is \(L^2(D_\varepsilon)\).

Using the stresses in \eqref{eq:unified-stresses}, define the bilinear forms
\begin{align}
 \mathfrak a_\varepsilon(u,v)
 &=\int_{\Omega_\varepsilon}
   [\lambda\diver u\,\diver v+2\mu\D u:\D v]
    +\int_{D_\varepsilon}2\widetilde\mu\D u:\D v,
 \label{eq:micro-a}\\
 b_\varepsilon(v,q)&=\int_{D_\varepsilon}q\,\diver v.
 \label{eq:micro-b}
\end{align}
The weak form of \eqref{eq:strong-stokes-inclusions} is to find
\((u_\varepsilon,p_\varepsilon)\in V\times L^2(D_\varepsilon)\) such that
\begin{equation}\label{eq:unified-mixed-problem}
 \begin{cases}
  \mathfrak a_\varepsilon(u_\varepsilon,v)+b_\varepsilon(v,p_\varepsilon)
       =\langle g,v\rangle_{\partial\Omega}&\forall v\in V,\\
  b_\varepsilon(u_\varepsilon,q)=0&\forall q\in L^2(D_\varepsilon).
 \end{cases}
\end{equation}
Integration by parts in each phase and continuity of the interface traction yield this weak formulation.
Because \(\D r=0\), \(\diver r=0\), and the load is compatible,
subtracting rigid displacements changes neither bilinear form nor the load.
Thus one may equivalently test over \(H^1(\Omega;\R^d)\).

\Needspace{6\baselineskip}
\begin{proposition}[Microscopic stability and well-posedness]
\label{thm:micro-full-pressure-infsup}\label{cor:micro-wellposedness}
There exists \(\beta=\beta(d,\Omega)>0\) such that,
for every \(0<\varepsilon<1\) and every \(q\in L^2(D_{\varepsilon})\),
\begin{equation}\label{eq:micro-full-pressure-infsup}
 \sup_{v\in V\setminus\{0\}}
 \frac{b_{\varepsilon}(v,q)}{\|v\|_V}
 \ge\beta\|q\|_{L^2(D_{\varepsilon})}.
\end{equation}

For \(g\in H_{\mathcal R}^{-1/2}(\partial\Omega)\), problem
\eqref{eq:unified-mixed-problem} has a unique solution
\((u_{\varepsilon},p_{\varepsilon})\in V\times L^2(D_{\varepsilon})\) satisfying
\[
 \|u_{\varepsilon}\|_{H^1(\Omega)}
 +\|p_{\varepsilon}\|_{L^2(D_{\varepsilon})}
 \le C\|g\|_{H^{-1/2}(\partial\Omega)}.
\]
Here \(C=C(d,\Omega,\lambda,\mu,\widetilde\mu)\) is independent of
\(\varepsilon\) and \(g\).
\end{proposition}

The proof of Proposition~\ref{thm:micro-full-pressure-infsup} is given in Appendix~\ref{app:mixed-stability-proof}.

\subsection{Cell problems and correctors}
\label{subsec:formal-cell-problems}

Cell problems describe the local displacement adjustment and inclusion pressure
induced by a prescribed macroscopic strain.
Let \(N\) be the outward normal to \(\omega\),
\(H^1_{\#,0}(Y;\R^d)\) the mean-zero periodic displacement space, and
\(\sigma_x,\sigma_y\) the stresses \eqref{eq:unified-stresses} in the corresponding variables.

To accommodate data at different expansion orders, consider the periodic cell problem
with body forces \(F_1,F_2\), divergence data \(f\), and interface traction data \(G\):
\begin{equation}\label{eq:general-periodic-cell}
 \begin{cases}
  \diver_y\sigma_y(u)=F_1
       &\text{in }Y\setminus\overline\omega,\\
  \diver_y\sigma_y(u,p)=F_2,\quad\diver_yu=f
       &\text{in }\omega,\\
  u|_+=u|_-
       &\text{on }\partial\omega,\\
  [\sigma_y(u,p)|_- -\sigma_y(u)|_+]N=G
       &\text{on }\partial\omega.
 \end{cases}
\end{equation}
The equations are understood in the periodic weak sense, interface tractions
as weak traces, and interface integrals as the corresponding duality pairings.

\begin{proposition}[Periodic cell well-posedness]
\label{thm:unified-cell-pair}\label{lem:periodic-full-phase-right-inverse}
Let \(F_1\in L^2(Y\setminus\overline\omega;\R^d)\),
\(F_2\in L^2(\omega;\R^d)\), \(f\in L^2(\omega)\), and
\(G\in H^{-1/2}(\partial\omega;\R^d)\). System \eqref{eq:general-periodic-cell}
has a unique solution \((u,p)\in H^1_{\#,0}(Y;\R^d)\times L^2(\omega)\) if and only if
\begin{equation}\label{eq:periodic-cell-compatibility}
 \int_{Y\setminus\overline\omega}F_1\,\mathrm dy
 +\int_\omega F_2\,\mathrm dy
 -\int_{\partial\omega}G\,\mathrm dS=0.
\end{equation}
In this case, the solution satisfies
\[
 \|u\|_{H^1(Y)}+\|p\|_{L^2(\omega)}
 \le C\bigl(\|F_1\|_{L^2(Y\setminus\overline\omega)}
              +\|F_2\|_{L^2(\omega)}
 +\|f\|_{L^2(\omega)}
              +\|G\|_{H^{-1/2}(\partial\omega)}\bigr),
\]
where \(C=C(d,\lambda,\mu,\widetilde\mu,\omega)\).
\end{proposition}
To motivate the cell problems at successive orders, set \(y=x/\varepsilon\) and formally
expand in the interior periodic structure, taking the coefficients sufficiently smooth
in the slow variable \(x\):
\[
 \begin{aligned}
 u_\varepsilon(x)&\sim u_0(x,y)+\varepsilon u_1(x,y)+\varepsilon^2u_2(x,y)+\cdots,\\
 p_\varepsilon(x)&\sim\varepsilon^{-1}p_{-1}(x,y)+p_0(x,y)+\varepsilon p_1(x,y)+\cdots.
 \end{aligned}
\]
The coefficients are periodic in \(y\), with the pressure coefficients defined in \(\omega\).
The leading-order terms give the cell equations for \((u_0,p_{-1})\):
\[
 \begin{cases}
  \diver_y\sigma_y(u_0)=0
       &\text{in }Y\setminus\overline\omega,\\
  \diver_y\sigma_y(u_0,p_{-1})=0,\quad\diver_yu_0=0
       &\text{in }\omega,\\
  u_0|_+=u_0|_-,\quad
  [\sigma_y(u_0,p_{-1})|_- -\sigma_y(u_0)|_+]N=0
       &\text{on }\partial\omega.
 \end{cases}
\]
Proposition~\ref{thm:unified-cell-pair} gives \(u_0=u_0(x)\) and \(p_{-1}=0\).
Since the material parameters are constant in each phase,
the next-order cell equations for \((u_1,p_0)\) reduce to
\[
 \begin{cases}
  \diver_y\sigma_y(u_1)=0
       &\text{in }Y\setminus\overline\omega,\\
  \diver_y\sigma_y(u_1,p_0)=0,\quad\diver_yu_1=-\diver_xu_0
       &\text{in }\omega,\\
  u_1|_+=u_1|_-
       &\text{on }\partial\omega,\\
  [\sigma_y(u_1,p_0)|_- -\sigma_y(u_1)|_+]N
       =[\sigma_x(u_0)-2\widetilde\mu\D_xu_0]N
       &\text{on }\partial\omega.
 \end{cases}
\]

For fixed \(x\), the data in these equations are determined solely by the macroscopic
strain \(\D_xu_0(x)\). By linearity, it suffices to solve the cell problem for each
elementary strain and then superpose the solutions.
Define the elementary affine displacements and their strains by
\begin{equation}\label{eq:basis-strain}
 p_{j}^{\beta}(y)
   =\tfrac12(y_\beta e_j+y_je_\beta),
 \qquad \mathsf E_{j}^{\beta}
   =\nabla_yp_{j}^{\beta}
     =\tfrac12(e_j\otimes e_\beta+e_\beta\otimes e_j).
\end{equation}
Here \(e_1,\ldots,e_d\) is the standard basis of \(\R^d\).

For each elementary strain \(\mathsf E_j^\beta\), let
\((\chi_j^\beta,r_j^\beta)\) denote the displacement and pressure correctors,
so that the total displacement \(p_j^\beta+\chi_j^\beta\), together with
pressure \(r_j^\beta\), satisfies the homogeneous Lam\'e--Stokes system within the cell.
All cell solutions are chosen as in Proposition~\ref{thm:unified-cell-pair}.
Linearity and uniqueness give the representation of the first-order terms
\begin{equation}\label{eq:first-order-cell-representation}
 u_1(x,y)=\chi_j^\beta(y)(\D_xu_0)_j^\beta(x),
 \qquad p_0(x,y)=r_j^\beta(y)(\D_xu_0)_j^\beta(x).
\end{equation}
Explicitly, the correctors satisfy (with stresses evaluated in the \(y\) variable)
\begin{equation}\label{eq:unified-cell-momentum}
 \begin{cases}
  \diver_y\sigma(\chi_j^\beta)=0
       &\text{in }Y\setminus\overline\omega,\\
  \diver_y\sigma(\chi_j^\beta,r_j^\beta)=0,\quad
  \diver_y\chi_j^\beta=-\delta_{j\beta}
       &\text{in }\omega,\\
  \chi_j^\beta|_+=\chi_j^\beta|_-
       &\text{on }\partial\omega,\\
  [\sigma(\chi_j^\beta,r_j^\beta)|_- -\sigma(\chi_j^\beta)|_+]N
   =[\lambda\delta_{j\beta}\mathbb I+2(\mu-\widetilde\mu)\mathsf E_j^\beta]N
       &\text{on }\partial\omega.
 \end{cases}
\end{equation}
The corrector equations \eqref{eq:unified-cell-momentum} satisfy the solvability condition
of Proposition~\ref{thm:unified-cell-pair}, namely \eqref{eq:periodic-cell-compatibility}.
Hence the correctors exist uniquely and satisfy
\begin{equation}\label{eq:cell-pair-energy-bound}
 \|\chi_{j}^{\beta}\|_{H^1(Y)}+\|r_{j}^{\beta}\|_{L^2(\omega)}\le C.
\end{equation}

The symmetry of the affine data and uniqueness give
\begin{equation}\label{eq:cell-corrector-symmetry}
 \chi_{j}^{\beta}=\chi_{\beta}^{j},\qquad r_{j}^{\beta}=r_{\beta}^{j}.
\end{equation}

\subsection{Effective equation}
\label{subsec:effective-equation}
\label{sec:unified-homogenization}

The macroscopic equation for \(u_0\) expresses equilibrium of the cell-averaged stress.
Collecting the \(\varepsilon^0\) terms in the momentum equations gives
the cell problem for \((u_2,p_1)\):
\[
 \begin{cases}
  \diver_y\sigma_y(u_2)
     =-\diver_y\sigma_x(u_1)
       -\diver_x[\sigma_x(u_0)+\sigma_y(u_1)]
       &\text{in }Y\setminus\overline\omega,\\
  \diver_y\sigma_y(u_2,p_1)
     =-\diver_y(2\widetilde\mu\D_xu_1)
       -\diver_x[2\widetilde\mu\D_xu_0+\sigma_y(u_1,p_0)]
       &\text{in }\omega,\\
  \diver_yu_2=-\diver_xu_1
       &\text{in }\omega,\\
  u_2|_+=u_2|_-
       &\text{on }\partial\omega,\\
  [\sigma_y(u_2,p_1)|_- -\sigma_y(u_2)|_+]N
     =[\sigma_x(u_1)|_+-2\widetilde\mu(\D_xu_1)|_-]N
       &\text{on }\partial\omega.
 \end{cases}
\]
Condition \eqref{eq:periodic-cell-compatibility} gives macroscopic equilibrium at this order.
Indeed, integrate the first two momentum equations over their respective phases and add them.
The outer cell-boundary terms cancel by periodicity, and the interface terms cancel
by the last traction condition, giving
\[
 \begin{aligned}
 &|Y\setminus\overline\omega|\diver_x\sigma_x(u_0)
   +\int_{Y\setminus\overline\omega}\diver_x\sigma_y(u_1)\,\mathrm dy\\
 &\quad+|\omega|\diver_x(2\widetilde\mu\D_xu_0)
   +\int_\omega\diver_x\sigma_y(u_1,p_0)\,\mathrm dy=0.
 \end{aligned}
\]

\begin{samepage}
Replace the integration
domains in \eqref{eq:micro-a} by \(Y\setminus\overline\omega\) and \(\omega\),
and take derivatives in \(y\). Denote the resulting cell bilinear form on
\(H^1(Y;\R^d)\) by \(\mathfrak a\).
Define the effective tensor \(\widehat A=(\widehat A_{ij}^{\alpha\beta})\) by
\begin{equation}\label{eq:effective-tensor-inclusions}
 \widehat A_{ij}^{\alpha\beta}
 =\mathfrak a(p_i^\alpha+\chi_i^\alpha,\,p_j^\beta+\chi_j^\beta)
  =\mathfrak a(p_i^\alpha,\,p_j^\beta+\chi_j^\beta)
   +\delta_{i\alpha}\int_\omega r_j^\beta.
\end{equation}
\end{samepage}
The second equality follows by testing the cell equation with \(\chi_i^\alpha\)
and using \(\diver_y\chi_i^\alpha=-\delta_{i\alpha}\) in \(\omega\).
Thus \(\widehat A\) maps a prescribed macroscopic strain to the average of the corresponding complete cell stress.
Substituting the representations of \(u_1,p_0\) from \eqref{eq:first-order-cell-representation}
into the macroscopic balance above gives the macroscopic equation in the interior.
Retaining the original traction data \(g\), we define the effective problem by
\begin{equation}\label{eq:unified-effective-problem}
 \begin{cases}
 \diver(\widehat A\nabla u_0)=0&x\in\Omega,\\
 \multicolumn{2}{l}{(\widehat A\nabla u_0)n=g,
 \quad u_0|_{\partial\Omega}\in H_{\mathcal R}^{1/2}(\partial\Omega).}
 \end{cases}
\end{equation}
Equivalently, we seek \(u_0\in V\) such that
\begin{equation}\label{eq:effective-weak-problem}
 \int_\Omega\widehat A\nabla u_0:\nabla v\,\mathrm dx
 =\langle g,v\rangle_{\partial\Omega}
 \qquad(v\in V).
\end{equation}
Only the inclusions are incompressible; the elastic matrix permits volumetric deformation.
The effective problem is therefore an elasticity system.
Cell pressure enters \(\widehat A\) through the averaged stress, without adding an independent macroscopic unknown.
The next proposition gives the tensor properties and well-posedness.

\begin{proposition}[Effective elasticity and well-posedness]
\label{thm:effective-tensor-properties}\label{thm:effective-tensor-ellipticity}
The effective tensor has the minor and major symmetries
\begin{equation}\label{eq:effective-tensor-symmetry}
 \widehat A_{ij}^{\alpha\beta}
 =\widehat A_{\alpha j}^{i\beta}
 =\widehat A_{i\beta}^{\alpha j}
 =\widehat A_{ji}^{\beta\alpha}.
\end{equation}
For every symmetric matrix \(\zeta=(\zeta_i^\alpha)\),
\begin{equation}\label{eq:effective-strong-ellipticity}
 \widehat A_{ij}^{\alpha\beta}\zeta_i^\alpha \zeta_j^\beta
 \ge c|\zeta|^2,
\end{equation}
where \(c=\min\{2\mu,d\lambda+2\mu,2\widetilde\mu\}>0\).
For every \(g\in H_{\mathcal R}^{-1/2}(\partial\Omega)\),
\eqref{eq:effective-weak-problem} has a unique solution \(u_0\in V\) satisfying
\[
 \|u_0\|_{H^1(\Omega)}\le C\|g\|_{H^{-1/2}(\partial\Omega)},
\]
where \(C=C(d,\Omega,\lambda,\mu,\widetilde\mu,\omega)\).
\end{proposition}
The cell estimate and definition \eqref{eq:effective-tensor-inclusions} give boundedness and the stated symmetries;
material positivity and the zero mean of the gradient of each periodic corrector give \eqref{eq:effective-strong-ellipticity}.
Korn's inequality and the Lax--Milgram theorem then yield well-posedness and the energy estimate for the effective problem.

\section{Local interface estimates and cell correctors}
\label{sec:unified-interface-regularity}

This section studies local interface regularity for the Lam\'e--Stokes transmission problem.
Take a ball centered at an inclusion interface point, translate and rescale it to \(B_1\),
and continue to write \(\omega\) for the transformed inclusion.
Consider the local system with \(G_\pm\) and \(f\):
\begin{equation}\label{eq:main-interface-system}
\begin{cases}
 \diver\sigma(u)=\diver G_+&\text{in }B_1\setminus\overline\omega,\\
 \diver\sigma(u,p)=\diver G_-,\quad\diver u=f&\text{in }B_1\cap\omega,\\
 u|_+=u|_-&\text{on }B_1\cap\partial\omega,\\
 (\sigma(u,p)-G_-)|_-\nu=(\sigma(u)-G_+)|_+\nu&\text{on }B_1\cap\partial\omega.
\end{cases}
\end{equation}
The stresses are defined by \eqref{eq:unified-stresses}, and \(\nu\) is the outward normal to the inclusion. Nonzero \(G_\pm,f\) accommodate frozen comparison and differentiation. The momentum equations and traction condition are understood weakly against \(H^1_0(B_1;\R^d)\), with no boundary condition on \(\partial B_1\).

\begin{theorem}[Local interface regularity]
\label{prop:physical-interface-regularity}\label{thm:main-interface-schauder}
Let \(k\ge1\) be an integer, \(0<\kappa<1\), and let the interface \(B_1\cap\partial\omega\) be of class \(C^{k,\kappa}\),
with \(G_+\in C^{k-1,\kappa}(\overline{B_1\setminus\overline\omega})\) and
\(G_-,f\in C^{k-1,\kappa}(\overline{B_1\cap\omega})\).
If \((u,p)\in H^1(B_1;\R^d)\times L^2(B_1\cap\omega)\) is a weak solution of
\eqref{eq:main-interface-system}, then
\begin{equation}\label{eq:main-interface-estimate}
\begin{aligned}
 &\|u\|_{C^{k,\kappa}(\overline{B_{1/2}\setminus\overline\omega})}
 +\|u\|_{C^{k,\kappa}(\overline{B_{1/2}\cap\omega})}
 +\|p\|_{C^{k-1,\kappa}(\overline{B_{1/2}\cap\omega})}\\
 &\quad\le C\Bigl\{\|u\|_{H^1(B_1)}+\|p\|_{L^2(B_1\cap\omega)}
 +\|G_+\|_{C^{k-1,\kappa}(\overline{B_1\setminus\overline\omega})}\\
 &\hspace{4em}+\|G_-\|_{C^{k-1,\kappa}(\overline{B_1\cap\omega})}
 +\|f\|_{C^{k-1,\kappa}(\overline{B_1\cap\omega})}\Bigr\}.
\end{aligned}
\end{equation}
The constant depends only on \(d,k,\kappa,\lambda,\mu,\widetilde\mu\)
and the geometry of this interface in the class \(C^{k,\kappa}\).
\end{theorem}

Flatten the local interface \(x_d=\psi(x')\), with the negative phase below the graph, using
\[
 \Phi(x)=(x',x_d+\psi(x')),\qquad a=(D\Phi)^{-1},\qquad\det D\Phi=1.
\]
After localizing and rescaling to \(B_1\), set \(B_r^\pm=B_r\cap\{\pm x_d>0\}\). Let \(A_\pm\) denote the weak pullbacks of the constitutive tensors. With \(G_\pm\) replaced by \((G_\pm\circ\Phi)a^T\) and \(u,p,f\) composed with \(\Phi\), retain the original notation and weak interpretation. System \eqref{eq:main-interface-system} becomes

\begin{equation}\label{eq:c1a-variable-system}
\begin{cases}
 \diver(A_+\nabla u)=\diver G_+&\text{in }B_1^+,\\
 \diver(A_-\nabla u+pa^T)=\diver G_-,\quad\diver(au)=f&\text{in }B_1^-,\\
 u|_+=u|_-&\text{on }B_1\cap\{x_d=0\},\\
 (A_-\nabla u+pa^T-G_-)|_-e_d
 =(A_+\nabla u-G_+)|_+e_d&\text{on }B_1\cap\{x_d=0\}.
\end{cases}
\end{equation}

These coefficients satisfy
\(A_\pm\in C^{k-1,\kappa}(\overline{B_1^\pm})\),
\(a\in C^{k-1,\kappa}(\overline{B_1})\), and \(a^{-1}\in L^\infty(B_1)\).
Their norm bounds and the constants \(C\) below depend only on \(d,k,\kappa,\lambda,\mu,\widetilde\mu\) and the interface geometry.
The coefficients also satisfy
\begin{equation}\label{eq:interface-structure}
 (A_\pm)_{ij}^{\alpha\beta}=(A_\pm)_{ji}^{\beta\alpha},\qquad
 \diver(a^T)=0,\qquad \partial_dA_\pm=0,\qquad\partial_da=0.
\end{equation}
Moreover, for every \(B_r=B_r(z)\Subset B_1\) with \(z_d=0\), we have
\begin{equation}\label{eq:schauder-chart-coercivity}
 \int_{B_r^+}A_+\nabla v:\nabla v+\int_{B_r^-}A_-\nabla v:\nabla v
 \ge C^{-1}\|\nabla v\|_{L^2(B_r)}^2,\qquad v\in H^1_0(B_r;\R^d).
\end{equation}
The transformation gives \eqref{eq:interface-structure}, and Korn's inequality after the change of variables gives \eqref{eq:schauder-chart-coercivity}. Freezing at any \(z\in B_1\cap\{x_d=0\}\) corresponds to an affine shear, so \eqref{eq:schauder-chart-coercivity} remains valid with a uniform constant.

\Needspace{4\baselineskip}
\begin{proposition}[Regularity estimate for the flattened interface]
\label{prop:interface-chart-estimate}
Let \(k\ge1\) be an integer, \(0<\kappa<1\), and let \(A_\pm,a\) be the coefficients obtained from the two constitutive laws above by a graph transformation of a \(C^{k,\kappa}\) interface.
Suppose that \(G_\pm\in C^{k-1,\kappa}(\overline{B_1^\pm};\R^{d\times d})\) and
\(f\in C^{k-1,\kappa}(\overline{B_1^-})\).
If \((u,p)\in H^1(B_1;\R^d)\times L^2(B_1^-)\) is a weak solution of system \eqref{eq:c1a-variable-system}, then
\begin{equation}\label{eq:interface-chart-estimate}
\begin{aligned}
 &\sum_{m\in\{+,-\}}\|u\|_{C^{k,\kappa}(\overline{B_{1/2}^m})}
 +\|p\|_{C^{k-1,\kappa}(\overline{B_{1/2}^-})}\\
 &\quad\le C\Bigl\{\|u\|_{H^1(B_1)}+\|p\|_{L^2(B_1^-)}
 +\sum_{m\in\{+,-\}}\|G_m\|_{C^{k-1,\kappa}(\overline{B_1^m})}
 +\|f\|_{C^{k-1,\kappa}(\overline{B_1^-})}\Bigr\}.
\end{aligned}
\end{equation}
The constant \(C\) depends only on \(d,k,\kappa,\lambda,\mu,\widetilde\mu\) and the geometry of the interface in the class \(C^{k,\kappa}\).
\end{proposition}

The proof of Proposition~\ref{prop:interface-chart-estimate} follows Zhuge's \cite[Section~3]{Zhuge2021} freezing comparison and excess-decay iteration.
The proof of Lemma~\ref{lem:c1a-joint-caccioppoli} is given in Appendix~\ref{app:interface-estimates}.

\subsection{Mixed estimates and affine approximation}
\label{subsec:interface-basic-estimate}\label{subsec:interface-energy-flat}\label{subsec:interface-affine-pairs}

The divergence right inverse below controls the full pressure and corrects cutoff test fields in the local energy estimate.

\begin{lemma}[Bogovski\u{\i}-type right inverse of the divergence]
\label{lem:c1a-rough-piola-lift}\label{lem:unified-rough-piola}
Let \(a\) be the matrix coefficient in system \eqref{eq:c1a-variable-system}.
There is \(r_0>0\) such that, whenever \(B_r(z)\Subset B_1\), \(z_d=0\), and \(0<r\le r_0\),
for every \(h\in L^2(B_r^-)\) there exists \(v\in H^1_0(B_r;\R^d)\) satisfying
\begin{equation}\label{eq:c1a-right-inverse}
 \diver(av)=h\quad\text{in }B_r^-,
 \qquad \|\nabla v\|_{L^2(B_r)}\le C\|h\|_{L^2(B_r^-)}.
\end{equation}
The constants \(r_0\) and \(C\) depend only on \(d,\kappa\), \(\|a\|_{C^{0,\kappa}(\overline{B_1})}\), and
\(\|a^{-1}\|_{L^\infty(B_1)}\).
If \(a\) is a constant matrix, the conclusion holds on every ball centered on \(\{x_d=0\}\), without a small-radius restriction, and \(C\) depends only on \(d\) and \(\|a^{-1}\|\).
\end{lemma}

\Needspace{4\baselineskip}
\begin{remark}[Inf--sup estimate for the full pressure]
\label{rem:c1a-full-pressure-infsup}
Under the assumptions of Lemma~\ref{lem:c1a-rough-piola-lift}, taking \(h=q\) in \eqref{eq:c1a-right-inverse} gives the full-pressure estimate for every \(q\in L^2(B_r^-)\)
\[
 \|q\|_{L^2(B_r^-)}
 \le C\sup_{0\ne\varphi\in H^1_0(B_r;\R^d)}
 \frac{\left|\int_{B_r^-}q\,\diver(a\varphi)\right|}
      {\|\nabla\varphi\|_{L^2(B_r)}}.
\]
No zero-mean condition is imposed on \(q\); when \(a\) is constant, no small-radius restriction is needed.
\end{remark}

\begin{lemma}[Local mixed energy estimate]
\label{lem:c1a-joint-caccioppoli}\label{lem:unified-interface-caccioppoli}
Let \(B_r(z)\Subset B_1\), \(z_d=0\), and \(0<r\le r_0\), where \(r_0\) is given by Lemma~\ref{lem:c1a-rough-piola-lift}.
Let \(G_\pm\in L^2(B_r^\pm;\R^{d\times d})\) and \(f\in L^2(B_r^-)\).
If \((u,p)\in H^1(B_r;\R^d)\times L^2(B_r^-)\) is a weak solution of system \eqref{eq:c1a-variable-system} in \(B_r\),
then, for every constant vector \(b\in\R^d\),

\begin{equation}\label{eq:c1a-joint-caccioppoli}
\begin{aligned}
 &\|\nabla u\|_{L^2(B_{r/2})}
       +\|p\|_{L^2(B_{r/2}^-)}\\
 &\qquad\le C\left(r^{-1}\|u-b\|_{L^2(B_r)}
             +\sum_{m\in\{+,-\}}\|G_m\|_{L^2(B_r^m)}
             +\|f\|_{L^2(B_r^-)}\right).
\end{aligned}
\end{equation}

\end{lemma}

We now turn to system \eqref{eq:c1a-variable-system} with constant coefficients. Following Zhuge \cite[Proposition~3.1 and Definition~3.2]{Zhuge2021}, we introduce continuous piecewise affine displacements and constant negative-phase pressures.

\begin{definition}[Piecewise affine displacements]
\label{def:interface-affine-pairs}
Let \(\mathcal L\) be the linear space of displacements of the form
\[
 \ell(x)=
 \begin{cases}
  b_0+M_+x,&x_d\ge0,\\
  b_0+M_-x,&x_d<0,
 \end{cases}
\]
satisfying
\[
 (M_+-M_-)e_t=0,\qquad t=1,\ldots,d-1
\]
where \(b_0\in\R^d\) and \(M_\pm\in\R^{d\times d}\).
For \(\ell\in\mathcal L\) and \(c\in\R\), call \((\ell,c)\) a piecewise affine pair, and define the upper-minus-lower traction difference map \(\mathcal T:\mathcal L\times\R\to\R^d\) by
\[
 \mathcal T(\ell,c):=(A_+M_+-A_-M_--ca^T)e_d.
\]
The space \(\mathcal L\) is independent of the coefficients, whereas \(\mathcal T\) depends on \(A_\pm,a\).
\end{definition}

A piecewise affine pair satisfies the homogeneous momentum equations in both phases and has negative-phase divergence \(\diver(a\ell)=a^T:M_-\).

\Needspace{10\baselineskip}
\begin{lemma}[Piecewise regularity for the homogeneous flat interface system]
\label{lem:c1a-flat-finite-provider}
Assume that the coefficients \(A_\pm,a\) in system \eqref{eq:c1a-variable-system} are constant.
If \((u,p)\in H^1(B_1;\mathbb R^d)\times L^2(B_1^-)\)
is a weak solution of this constant-coefficient system with \(G_\pm=0,\ f=0\), then, for every integer \(n\ge1\),
\begin{equation}\label{eq:c1a-flat-finite-provider}
 \sum_{m\in\{+,-\}}\|u\|_{H^{n+1}(B_{1/2}^m)}
       +\|p\|_{H^n(B_{1/2}^-)}
 \le C_n\|u\|_{L^2(B_1)}.
\end{equation}
The constant \(C_n\) depends only on \(d,n\), the coercivity constant, and the norm bounds for \(A_\pm,a,a^{-1}\).
\end{lemma}

Write \(A_\pm^{dd}=((A_\pm)_{ij}^{dd})_{i,j=1}^d\). The coefficient matrices for recovering normal derivatives in system \eqref{eq:c1a-variable-system} are
\begin{equation}\label{eq:interface-normal-matrices}
 A_+^{dd},\qquad
 \begin{pmatrix}A_-^{dd}&a^Te_d\\ e_d^Ta&0\end{pmatrix}.
\end{equation}
Coercivity, the coefficient bounds, and boundedness of \(a^{-1}\) give
\begin{equation}\label{eq:interface-normal-acoustic-blocks}
 C^{-1}\mathbb I\le A_\pm^{dd}\le C\mathbb I,
 \qquad e_d^Ta(A_-^{dd})^{-1}a^Te_d\ge C^{-1}.
\end{equation}
Both matrices are therefore uniformly invertible.

\phantomsection\label{app:flat-regularity-proof}
By standard tangential difference-quotient arguments based on Lemmas~\ref{lem:c1a-rough-piola-lift} and \ref{lem:c1a-joint-caccioppoli}, followed by normal-derivative recovery using \eqref{eq:interface-normal-matrices}--\eqref{eq:interface-normal-acoustic-blocks}, we obtain \eqref{eq:c1a-flat-finite-provider}.

Assume that the coefficients \(A_\pm,a\) in system \eqref{eq:c1a-variable-system} are constant, and let \(G_\pm\in\R^{d\times d}\) and \(f\in\R\) be constant data. Consider the system in \(B_1\)
\begin{equation}\label{eq:c1a-frozen-system}
\begin{cases}
 \diver(A_+\nabla u)=0&\text{in }B_1^+,\\
 \diver(A_-\nabla u+pa^T)=0,\quad\diver(au)=f&\text{in }B_1^-,\\
 u|_+=u|_-&\text{on }B_1\cap\{x_d=0\},\\
 (A_-\nabla u+pa^T-G_-)|_-e_d=(A_+\nabla u-G_+)|_+e_d
 &\text{on }B_1\cap\{x_d=0\}.
\end{cases}
\end{equation}
\Needspace{4\baselineskip}
\begin{lemma}[Regularity for constant data]
\label{lem:c1a-flat-constant-data}
If \((u,p)\in H^1(B_1;\R^d)\times L^2(B_1^-)\) is a weak solution of system \eqref{eq:c1a-frozen-system},
then, for each integer \(n\ge1\),
\[
 \sum_{m\in\{+,-\}}\|u\|_{H^{n+1}(B_{1/2}^m)}
       +\|p\|_{H^n(B_{1/2}^-)}
 \le C_n\bigl(\|u\|_{L^2(B_1)}+|G_+|+|G_-|+|f|\bigr).
\]
The constant \(C_n\) depends only on \(d,n\), the coercivity constant, and the norm bounds for \(A_\pm,a,a^{-1}\).
\end{lemma}
\phantomsection\label{app:constant-data-proof}
By \eqref{eq:interface-normal-matrices}--\eqref{eq:interface-normal-acoustic-blocks}, choose \((\ell,c)\in\mathcal L\times\R\) such that
\[
\begin{gathered}
 \diver(a\ell)=f\quad\text{in }B_1^-,\qquad\mathcal T(\ell,c)=(G_+-G_-)e_d,\\
 \sum_\pm\|\ell\|_{H^{n+1}(B_1^\pm)}+|c|\le C_n(|G_+|+|G_-|+|f|).
\end{gathered}
\]
Applying Lemma~\ref{lem:c1a-flat-finite-provider} to the homogeneous difference \((u-\ell,p-c)\) gives the result.

By Lemma~\ref{lem:c1a-flat-constant-data} and Sobolev embedding, the two phasewise gradients and the lower-phase pressure have one-sided limits at the interface. We approximate \((u,p)\) by its Taylor affine pair at the origin.
\Needspace{4\baselineskip}
\begin{lemma}[Piecewise affine approximation and excess decay]
\label{lem:c1a-flat-decay}
Let \(A_\pm,a\) be constant coefficients, let \(G_\pm,f\) be constant data, and let \((u,p)\in H^1(B_1;\mathbb R^d)\times L^2(B_1^-)\) be a weak solution of \eqref{eq:c1a-frozen-system}. Define
\[
 \ell(x)=u(0)+\nabla u_+(0)x\,\mathbf1_{\{x_d\ge0\}}
                  +\nabla u_-(0)x\,\mathbf1_{\{x_d<0\}},\qquad c=p_-(0).
\]
Then \((\ell,c)\in\mathcal L\times\R\) satisfies \(\diver(a\ell)=f\) in \(B_1^-\) and \(\mathcal T(\ell,c)=(G_+-G_-)e_d\). For every \(0<\theta\le1/8\),
\begin{equation}\label{eq:c1a-flat-excess-decay}
 \int_{B_\theta}|u-\ell|^2
 \le C\theta^{d+4}\inf_{(\ell_0,c_0)}\int_{B_1}|u-\ell_0|^2.
\end{equation}
The infimum is taken over pairs satisfying \(\diver(a\ell_0)=f\) in \(B_1^-\) and \(\mathcal T(\ell_0,c_0)=(G_+-G_-)e_d\), with \((\ell_0,c_0)\in\mathcal L\times\R\). The constant \(C\) depends only on \(d\), the coercivity constant, and the bounds for \(A_\pm,a,a^{-1}\).
\end{lemma}
\begin{proof}
Lemma~\ref{lem:c1a-flat-constant-data} allows one-sided limits at the origin in the continuity, divergence, and traction conditions of \eqref{eq:c1a-frozen-system}. Thus \((\ell,c)\) satisfies the stated conditions.
For any \((\ell_0,c_0)\) in the infimum in \eqref{eq:c1a-flat-excess-decay}, the pair \((u-\ell_0,p-c_0)\) solves \eqref{eq:c1a-frozen-system} with \(G_\pm=f=0\). Estimate \eqref{eq:c1a-flat-finite-provider} and Sobolev embedding give
\[
 \sum_{m\in\{+,-\}}\|\nabla^2u\|_{L^\infty(B_{1/4}^m)}
 \le C\|u-\ell_0\|_{L^2(B_1)}.
\]
Taylor's formula gives \(|u(x)-\ell(x)|\le C|x|^2\|u-\ell_0\|_{L^2(B_1)}\) in \(B_{1/8}^\pm\). Squaring, integrating, and taking the infimum over \((\ell_0,c_0)\) proves \eqref{eq:c1a-flat-excess-decay}.
\end{proof}

\subsection{Freezing comparison and iteration}

\label{subsec:interface-perturbation}

\begin{proof}[Proof of Proposition~\ref{prop:interface-chart-estimate}]

First take \(k=1\). Choose \(r_0<1/24\) with \(4r_0\) in the radius range of Lemma~\ref{lem:c1a-rough-piola-lift}; \(r_0\) may be reduced below. We work at the origin, with constants uniform for interface centers in \(B_{3/4}\cap\{x_d=0\}\). It suffices to prove, for centers in \(B_{5/8}\cap\{x_d=0\}\) and \(0<r\le r_0\),

\begin{equation}\label{eq:c1a-affine-target}
 \inf_{\ell\in\mathcal L}\Bigl(\fint_{B_r}|u-\ell|^2\Bigr)^{1/2}
 \le Cr^{1+\kappa}\Bigl(\sum_\pm\|G_\pm\|_{C^\kappa(B_1^\pm)}
 +\|f\|_{C^\kappa(B_1^-)}+\|u\|_{L^2(B_1)}\Bigr).
\end{equation}
By the Campanato argument and uniqueness of the \(H^1\) trace, \eqref{eq:c1a-affine-target} gives a common \(C^{1,\kappa}\) interface trace whose norm is controlled by the parenthesized quantity in that estimate. Using this trace as Dirichlet data in physical coordinates, Lemma~\ref{lem:single-phase-boundary-schauder} and the interior estimates \cite[Theorems~5.19--5.20]{GiaquintaMartinazzi2012} and \cite[Theorem~4.4]{NguyenStara2007} yield the case \(k=1\), with lower-order terms controlled by the right-hand side of \eqref{eq:interface-chart-estimate}.

By linearity, we may normalize the solution and the data so that
\begin{equation}\label{eq:c1a-data-normalization}
 \sum_\pm\|G_\pm\|_{C^\kappa(B_1^\pm)}
 +\|f\|_{C^\kappa(B_1^-)}+\|u\|_{L^2(B_1)}\le1.
\end{equation}
We prove the estimate under this normalization; the general case follows by scaling.

\smallskip\noindent\textbf{Step 1: Freezing comparison.}
For \(0<r\le r_0\), let \((V,Q)\in H^1(B_r;\R^d)\times L^2(B_r^-)\) be a weak solution of the following frozen transmission problem (obtained by scaling \eqref{eq:c1a-frozen-system} to \(B_r\)):

\begin{equation}\label{eq:c1a-frozen-dirichlet-system}
\begin{cases}
 \diver(A_+(0)\nabla V)=0&\text{in }B_r^+,\\
 \diver(A_-(0)\nabla V+Qa(0)^T)=0,\quad\diver(a(0)V)=f(0)&\text{in }B_r^-,\\
 V=u&\text{on }\partial B_r,\\
 V|_+=V|_-&\text{on }B_r\cap\{x_d=0\},\\
 \begin{aligned}[b]
  &(A_-(0)\nabla V+Qa(0)^T-G_-(0))|_-e_d\\
  &\quad=(A_+(0)\nabla V-G_+(0))|_+e_d
 \end{aligned}
 &\text{on }B_r\cap\{x_d=0\}.
\end{cases}
\end{equation}

Here \(V-u\in H^1_0(B_r;\R^d)\). Coercivity \eqref{eq:schauder-chart-coercivity}, Lemma~\ref{lem:c1a-rough-piola-lift}, and Remark~\ref{rem:c1a-full-pressure-infsup} give existence and uniqueness for this mixed problem.

Subtracting the weak forms of systems \eqref{eq:c1a-variable-system} and \eqref{eq:c1a-frozen-dirichlet-system}, we obtain, for every \(\varphi\in H^1_0(B_r;\R^d)\),
\begin{equation}\label{eq:c1a-difference-weak-form}
\begin{aligned}
 &\sum_\pm\int_{B_r^\pm}A_\pm(0)\nabla(u-V):\nabla\varphi
   +\int_{B_r^-}(p-Q)\diver(a(0)\varphi)\\
 &=\sum_\pm\int_{B_r^\pm}
   \bigl[G_\pm-G_\pm(0)+(A_\pm(0)-A_\pm)\nabla u\bigr]:\nabla\varphi
   +\int_{B_r^-}p(a(0)-a)^T:\nabla\varphi.
\end{aligned}
\end{equation}
The divergence constraint gives
\[
 \diver(a(0)(u-V))=f-f(0)+(a(0)-a)^T:\nabla u\quad\text{in }B_r^-.
\]
By the constant-matrix case of Lemma~\ref{lem:c1a-rough-piola-lift}, choose \(v\in H^1_0(B_r;\R^d)\) satisfying
\[
 \diver(a(0)v)=\diver(a(0)(u-V))\ \text{in }B_r^-,\quad
 \|\nabla v\|_{L^2(B_r)}\le C\|\diver(a(0)(u-V))\|_{L^2(B_r^-)}.
\]
\noindent\begin{minipage}{\linewidth}
Testing \eqref{eq:c1a-difference-weak-form} with \(u-V-v\) eliminates the pressure term. Coercivity \eqref{eq:schauder-chart-coercivity}, the full-pressure estimate of Remark~\ref{rem:c1a-full-pressure-infsup}, H\"older continuity of the coefficients and data, and \eqref{eq:c1a-data-normalization} give

\begin{equation}\label{eq:c1a-frozen-comparison}
 \left(\fint_{B_r}|\nabla(u-V)|^2+\fint_{B_r^-}|p-Q|^2\right)^{1/2}
 \le Cr^\kappa\left\{
 \left(\fint_{B_r}|\nabla u|^2+\fint_{B_r^-}|p|^2\right)^{1/2}+1\right\}.
\end{equation}
\end{minipage}

\smallskip\noindent\textbf{Step 2: Local energy decay.}
Write

\begin{equation}\label{eq:c1a-energy-functional}
 \Psi(r)=\int_{B_r}|\nabla u|^2+\int_{B_r^-}|p|^2+r^d.
\end{equation}

Apply the mixed energy estimate \eqref{eq:c1a-joint-caccioppoli} on \(B_{4r_0}\). The normalization condition \eqref{eq:c1a-data-normalization} gives \(\Psi(r_0)\le C\).
Apply Lemma~\ref{lem:c1a-flat-constant-data}, the Poincar\'e inequality, and Sobolev embedding to \(V-\fint_{B_r}V\). Scaling and \eqref{eq:c1a-frozen-comparison} yield
\[
  \Psi(s)\le C\left\{\left(\frac{s}{r}\right)^d+r^{2\kappa}\right\}\Psi(r),
  \qquad 0<s\le r/2.
\]

Choose \(\theta\in(0,1/8)\) and \(r_0\) sufficiently small so that \(\Psi(\theta r)\le\theta^{d-\kappa}\Psi(r)\). Iteration and the monotonicity of \(\Psi\) give

\begin{equation}\label{eq:c1a-energy-decay}
 \Psi(r)\le Cr^{d-\kappa},\qquad 0<r\le r_0.
\end{equation}

\smallskip\noindent\textbf{Step 3: \(C^{1,\kappa}\) estimates.}
In the map \(\mathcal T\) of Definition~\ref{def:interface-affine-pairs}, use \(A_\pm(0),a(0)\). Define
\begin{equation}\label{eq:c1a-excess}
 \mathcal E(r)=\inf_{(\ell,c)}\int_{B_r}|u-\ell|^2,
\end{equation}
where \((\ell,c)\in\mathcal L\times\R\) satisfies
\[
 \diver(a(0)\ell)=f(0)\quad\text{in }\{x_d<0\},\qquad
 \mathcal T(\ell,c)=(G_+(0)-G_-(0))e_d.
\]
The \hyperref[app:constant-data-proof]{affine construction following Lemma~\ref*{lem:c1a-flat-constant-data}}, applied to the frozen coefficients and data, and normalization \eqref{eq:c1a-data-normalization} give \(\mathcal E(r_0)\le C\). The admissibility conditions are independent of the radius, so \(\mathcal E\) is nondecreasing.

By \eqref{eq:c1a-frozen-comparison}, the definition \eqref{eq:c1a-energy-functional}, and the Poincar\'e inequality for \(u-V\in H^1_0(B_r;\R^d)\),
\[
 \int_{B_r}|u-V|^2\le Cr^2\int_{B_r}|\nabla(u-V)|^2
       \le Cr^{2+2\kappa}\Psi(r).
\]
Combining the scaled form of \eqref{eq:c1a-flat-excess-decay} with the triangle inequality and taking the infimum as defined in \eqref{eq:c1a-excess} gives
\begin{equation}\label{eq:c1a-one-step-psi}
 \mathcal E(s)\le C\left(\frac{s}{r}\right)^{d+4}\mathcal E(r)
             +Cr^{2+2\kappa}\Psi(r),\qquad 0<s\le r/8.
\end{equation}
By monotonicity, after increasing \(C\), \eqref{eq:c1a-one-step-psi} holds for \(0<s\le r\). Combining \eqref{eq:c1a-energy-decay} with the decay lemma \cite[Lemma~3.7]{Zhuge2021} yields
\[
 \mathcal E(r)\le Cr^{d+2+\kappa},\qquad 0<r\le r_0.
\]
The trace argument above with exponent \(\kappa/2\), followed by Lemma~\ref{lem:single-phase-boundary-schauder} in physical coordinates and Lemma~\ref{lem:c1a-joint-caccioppoli} for the lower-order terms, gives uniform bounds for \(u\) on \(\overline{B_{r_0}^\pm(z)}\) in \(C^{1,\kappa/2}\), \(z\in B_{5/8}\cap\{x_d=0\}\). Remark~\ref{rem:c1a-full-pressure-infsup} and the weak form of \eqref{eq:c1a-variable-system} control the full pressure, yielding

\begin{equation}\label{eq:c1a-improved-energy-decay}
 \Psi(r)=\int_{B_r}|\nabla u|^2+\int_{B_r^-}|p|^2+r^d
 \le Cr^d,\qquad 0<r\le r_0.
\end{equation}
Substitute \eqref{eq:c1a-improved-energy-decay} into \eqref{eq:c1a-one-step-psi}. Since \(d+4>d+2+2\kappa\), applying the decay lemma \cite[Lemma~3.7]{Zhuge2021} once more yields
\[
 \mathcal E(r)\le Cr^{d+2+2\kappa},\qquad 0<r\le r_0.
\]
This gives \eqref{eq:c1a-affine-target} and completes the case \(k=1\).

\smallskip\noindent\textbf{Higher-order estimates.}
For \(k\ge2\), argue by induction on \(k\). First consider tangential derivatives.
The conclusion at order \(k-1\) bounds \(\nabla u\) on \(B_{3/4}^\pm\) and \(p\) on \(B_{3/4}^-\) in \(C^{k-2,\kappa}\). Lemma~\ref{lem:c1a-joint-caccioppoli} and the interior estimates bound the tangential difference quotients on \(B_{2/3}\) uniformly in \(H^1\times L^2\). Passing to the limit in the step size, for \(t<d\), the pair \((\partial_tu,\partial_tp)\) satisfies system \eqref{eq:c1a-variable-system} with data

\[
\begin{aligned}
 &\partial_tG_m-(\partial_tA_m)\nabla u
             -\mathbf1_{m=-}p(\partial_ta)^T,\qquad m\in\{+,-\},\\
 &\partial_tf-(\partial_ta)^T:\nabla u.
\end{aligned}
\]

The induction hypothesis and the estimate at order \(k-1\) give
\[
 \partial_tu\in C^{k-1,\kappa}(\overline{B_{5/8}^\pm}),\qquad
 \partial_tp\in C^{k-2,\kappa}(\overline{B_{5/8}^-}),\qquad t<d.
\]
These norms are controlled by the right-hand side of \eqref{eq:interface-chart-estimate}. Apply \(\partial_d^{k-2}\) to the momentum equations and \(\partial_d^{k-1}\) to the negative-phase divergence equation. By \eqref{eq:interface-structure}, \eqref{eq:interface-normal-matrices}--\eqref{eq:interface-normal-acoustic-blocks}, and the \(C^{0,\kappa}\) bounds for the inverse matrices, this controls \(\partial_d^ku\) in both phases and \(\partial_d^{k-1}p\) in the negative phase. Together with the bound for \(u\) in \(L^2\), this gives \eqref{eq:interface-chart-estimate}.

\end{proof}

Proposition~\ref{prop:interface-chart-estimate} in interface charts, the inverse coordinate changes, and interior estimates give Theorem~\ref{thm:main-interface-schauder} by a finite covering.

For \(\partial\omega\in C^{1,\kappa}\), \(0<\kappa<1\), apply Theorem~\ref{thm:main-interface-schauder} with \(k=1\) to the periodic extensions of the correctors in \eqref{eq:unified-cell-momentum}. Together with \eqref{eq:cell-pair-energy-bound}, interior estimates and a finite covering, this gives

\begin{equation}\label{eq:unified-bounded-correctors}
 \|\chi_j^\beta\|_{W^{1,\infty}(Y)}
 +\|r_j^\beta\|_{L^\infty(\omega)}\le C,
 \qquad 1\le j,\beta\le d.
\end{equation}
The constant \(C\) depends only on \(d,\kappa,\lambda,\mu,\widetilde\mu,\omega\).

\section{Quantitative Homogenization}
\label{sec:unified-quantitative-architecture}

We prove the convergence estimates in Theorem~\ref{thm:main-quantitative} and the optimality of the displacement convergence order.

\subsection{First-order approximation and residual estimates}
\label{sec:unified-smoothed-rate}

Following Shen and Zhuge \cite[Section~3]{ShenZhuge2017Elasticity}, we construct a smoothed first-order approximation from the cell correctors and effective tensor in Section~\ref{sec:unified-cell}, accounting for the inclusion divergence residual and the elastic boundary buffer.

Denote the periodic coefficients for the displacement part of the stresses in \eqref{eq:unified-stresses} by \(A=(A_{ij}^{\alpha\beta})\). On \(Y\), they are
\begin{equation}\label{eq:periodic-orientation-tensor}
 A_{ij}^{\alpha\beta}(y)
 =\lambda\mathbf1_{Y\setminus\overline\omega}(y)
          \delta_{i\alpha}\delta_{j\beta}
 +\bigl[\mu\mathbf1_{Y\setminus\overline\omega}(y)
         +\widetilde\mu\mathbf1_\omega(y)\bigr]
   (\delta_{ij}\delta_{\alpha\beta}+\delta_{i\beta}\delta_{j\alpha}).
\end{equation}
Extend \(A\) periodically to \(\R^d\).

Extend the cell pressure \(r_{j}^{\beta}\) by zero from \(\omega\) to \(Y\), then periodically, retaining its notation.
Denote the difference between the corresponding cell stress and its mean by \(\mathcal B\):
\begin{equation}\label{eq:unified-flux-discrepancy}
 \mathcal B_{ij}^{\alpha\beta}(y)
 =A_{ij}^{\alpha\beta}(y)
  +A_{ik}^{\alpha\gamma}(y)\partial_{y_\gamma}(\chi_j^\beta)^k(y)
  +r_{j}^{\beta}(y)\delta_{i\alpha}
  -\widehat A_{ij}^{\alpha\beta}.
\end{equation}

The cell equation \eqref{eq:unified-cell-momentum} gives \(\partial_{y_\alpha}\mathcal B_{ij}^{\alpha\beta}=0\) in periodic distributions, while the averaged-stress formula \eqref{eq:effective-tensor-inclusions} gives \(\int_Y\mathcal B_{ij}^{\alpha\beta}=0\).
Using the cell energy bound \eqref{eq:cell-pair-energy-bound}, the standard periodic Poisson construction \cite[Proposition~3.1.1]{Shen2018} then yields a mean-zero periodic potential \(\phi_{ij}^{\gamma\alpha\beta}\in H^1_{\#,0}(Y)\) satisfying
\begin{equation}\label{eq:unified-flux-potential}
 \begin{gathered}
 \mathcal B_{ij}^{\alpha\beta}
 =\partial_{y_\gamma}\phi_{ij}^{\gamma\alpha\beta},
 \qquad
 \phi_{ij}^{\gamma\alpha\beta}=-\phi_{ij}^{\alpha\gamma\beta},
 \qquad
 \|\phi\|_{H^1(Y)}\le C,\\
 \|\phi\|_{L^\infty(Y)}\le C
 \qquad\text{if }\partial\omega\in C^{1,\kappa},\quad 0<\kappa<1.
 \end{gathered}
\end{equation}
The latter bound follows from \eqref{eq:unified-bounded-correctors} and the periodic Poisson \(W^{2,p}\) estimate (\(p>d\)); its constant \(C\) depends only on \(d,\kappa,\lambda,\mu,\widetilde\mu,\omega\).

Extend \(u_0\) to \(\R^d\), retaining its notation and with support in a fixed bounded set, using the \(W^{2,p}\)-bounded operator in \cite[Chapter~VI, Theorem~5]{Stein1970}, \(2\le p<\infty\).
Choose nonnegative \(\rho\in C_c^\infty(B_{1/2})\) with \(\int_{\R^d}\rho=1\), and define
\begin{equation}\label{eq:unified-smoothing-definition}
 S_\varepsilon f(x)=\int_{\R^d}\rho(y)f(x-\varepsilon y)\,\mathrm dy,\qquad
 S_\varepsilon^2=S_\varepsilon\circ S_\varepsilon.
\end{equation}
For \(f\in H^1(\R^d)\), the standard smoothing estimate gives
\begin{equation}\label{eq:unified-smoothing-L2-bound}
 \|f-S_\varepsilon^2f\|_{L^2(\R^d)}
 \le2\varepsilon\|\nabla f\|_{L^2(\R^d)}.
\end{equation}

Denote the boundary layer and the buffer outside the retained complete cells by
\begin{equation}\label{eq:unified-boundary-regions}
 \Omega(t)
 =\{x\in\Omega:\operatorname{dist}(x,\partial\Omega)<t\},
 \quad t>0,
 \qquad K_\varepsilon
 =\Omega\setminus\bigcup_{k\in\Pi_\varepsilon}\overline{Y_\varepsilon^k}.
\end{equation}
By Section~\ref{sec:unified-model}, \(K_\varepsilon\) lies entirely in the elastic phase
and within distance \(\sqrt d\,\varepsilon\) of the outer boundary.

The boundary-layer estimate \cite[Proposition~3.1.7]{Shen2018} gives, for \(F\in H^1(\Omega)\),
\begin{equation}\label{eq:unified-boundary-layer-estimate}
 \|F\|_{L^2(\Omega(t))}
 \le Ct^{1/2}\|F\|_{H^1(\Omega)},\qquad t>0.
\end{equation}
Let \(b\in L^2(Y)\) be \(Y\)-periodic and \(0<\varepsilon<1\). The periodic-product estimate \cite[Proposition~3.1.5]{Shen2018} gives, for \(F\in L^2(\R^d)\),
\begin{equation}\label{eq:unified-periodic-smoothing-product}
 \|b(x/\varepsilon)S_\varepsilon^2F\|_{L^2(\R^d)}
 \le C\|b\|_{L^2(Y)}\|F\|_{L^2(\R^d)}.
\end{equation}
If in addition \(F\in H^1(\R^d)\), then for fixed \(c>0\), we also have
\begin{equation}\label{eq:unified-periodic-smoothing-layer}
 \|b(x/\varepsilon)S_\varepsilon^2F\|_{L^2(\Omega(c\varepsilon))}
 \le C\varepsilon^{1/2}\|b\|_{L^2(Y)}\|F\|_{H^1(\R^d)}.
\end{equation}
Fix \(c>\sqrt d+3\) and choose \(\eta_\varepsilon\in C_c^\infty(\Omega)\) satisfying
\begin{equation}\label{eq:unified-cutoff-properties}
 \begin{aligned}
 &0\le\eta_\varepsilon\le1,\qquad |\nabla\eta_\varepsilon|\le C\varepsilon^{-1},\\
 &\eta_\varepsilon=0\quad\text{on }\Omega((\sqrt d+2)\varepsilon),
 \qquad
 \eta_\varepsilon=1\quad\text{on }\Omega\setminus\Omega(c\varepsilon).
 \end{aligned}
\end{equation}

Let \((u_\varepsilon,p_\varepsilon)\in V\times L^2(D_\varepsilon)\) and \(u_0\in V\) satisfy \eqref{eq:unified-mixed-problem} and \eqref{eq:effective-weak-problem}, respectively, and define the first-order approximation errors
\begin{align}
 w_\varepsilon
 &=u_\varepsilon-u_0
   -\varepsilon\chi_{j}^{\beta}(x/\varepsilon)
      \eta_\varepsilon S_\varepsilon^2((\D u_0)_j^\beta),
 \label{eq:smoothed-displacement-corrector}\\
 q_\varepsilon
 &=p_\varepsilon
   -r_{j}^{\beta}(x/\varepsilon)
      \eta_\varepsilon S_\varepsilon^2((\D u_0)_j^\beta)
       ,\qquad x\in D_\varepsilon.
 \label{eq:smoothed-pressure-corrector}
\end{align}
By \eqref{eq:unified-cutoff-properties}, \(w_\varepsilon\in V\); below we use the bilinear forms in \eqref{eq:micro-a}--\eqref{eq:micro-b}.

\medskip\noindent\begin{minipage}[b]{\linewidth}
\begin{lemma}[Residual estimate]
\label{lem:unified-residual-bound}
Suppose that \(u_0\in H^2(\Omega;\R^d)\), and let \(w_\varepsilon,q_\varepsilon\) be defined by
\eqref{eq:smoothed-displacement-corrector}--\eqref{eq:smoothed-pressure-corrector}.
Then, for every \(0<\varepsilon<1\), \(\varphi\in V\), and \(h\in L^2(D_\varepsilon)\),
\begin{align}
 \left|\mathfrak a_\varepsilon(w_\varepsilon,\varphi)
       +b_\varepsilon(\varphi,q_\varepsilon)\right|
 &\le C\|u_0\|_{H^2(\Omega)}
 \left(\varepsilon\|\varphi\|_{H^1(\Omega)}
 +\varepsilon^{1/2}
   \|\nabla\varphi\|_{L^2(\Omega(c\varepsilon))}\right),
 \label{eq:unified-F-bulk-boundary-split}\\
 |b_\varepsilon(w_\varepsilon,h)|
 &\le C\|u_0\|_{H^2(\Omega)}
 \left(\varepsilon\|h\|_{L^2(D_\varepsilon)}
 +\varepsilon^{1/2}
   \|h\|_{L^2(D_\varepsilon\cap\Omega(c\varepsilon))}\right).
 \label{eq:unified-G-bulk-boundary-split}
\end{align}
Here \(C\) is independent of \(\varepsilon\) and \(u_0\).
\end{lemma}
\end{minipage}

\begin{proof}
For \(\varphi\in V\), subtract \eqref{eq:effective-weak-problem} from \eqref{eq:unified-mixed-problem}, cancel the common right-hand side, and substitute \eqref{eq:smoothed-displacement-corrector}--\eqref{eq:smoothed-pressure-corrector}.
Using \eqref{eq:unified-flux-discrepancy}--\eqref{eq:unified-flux-potential}, weak integration by parts as in \cite[proof of Lemma~3.5]{ShenZhuge2017Elasticity} gives
\begin{equation}\label{eq:unified-momentum-residual-expansion}
\begin{aligned}
 \mathfrak a_\varepsilon(w_\varepsilon,\varphi)
  +b_\varepsilon(\varphi,q_\varepsilon)
 ={}&\int_\Omega(\widehat A-A(x/\varepsilon))
 [\D u_0-\eta_\varepsilon S_\varepsilon^2(\D u_0)]:\D\varphi\\
 &+\varepsilon\int_\Omega
 [\phi_{ij}^{\gamma\alpha\beta}
 -A_{ik}^{\alpha\gamma}(\chi_j^\beta)^k](x/\varepsilon)\\
 &\hspace{4em}\times
 \partial_\gamma[\eta_\varepsilon S_\varepsilon^2((\D u_0)_j^\beta)]
 \partial_\alpha\varphi^i\\
 &-\int_{K_\varepsilon}
 [\lambda(\diver u_0)\mathbb I+2\mu\D u_0-A(x/\varepsilon)\D u_0]:\D\varphi .
\end{aligned}
\end{equation}

Decompose the strain difference in the first term on the right of \eqref{eq:unified-momentum-residual-expansion} as
\begin{equation}\label{eq:unified-strain-splitting}
 \D u_0-\eta_\varepsilon S_\varepsilon^2(\D u_0)
 =(\D u_0-S_\varepsilon^2(\D u_0))
  +(1-\eta_\varepsilon)S_\varepsilon^2(\D u_0).
\end{equation}
By \eqref{eq:unified-smoothing-L2-bound} and \eqref{eq:unified-periodic-smoothing-product}, the parts of the first two terms in \eqref{eq:unified-momentum-residual-expansion} free of \(1-\eta_\varepsilon\) and \(\nabla\eta_\varepsilon\) are bounded by the first term on the right of \eqref{eq:unified-F-bulk-boundary-split}.

The remaining parts are supported in \(\Omega(c\varepsilon)\) and give the second term by \eqref{eq:unified-periodic-smoothing-layer}--\eqref{eq:unified-cutoff-properties}. On \(K_\varepsilon\subset\Omega(c\varepsilon)\), \(\eta_\varepsilon=\nabla\eta_\varepsilon=0\) by \eqref{eq:unified-boundary-regions}, \eqref{eq:unified-cutoff-properties}, so \eqref{eq:unified-boundary-layer-estimate} bounds the last integral in the same way. This proves \eqref{eq:unified-F-bulk-boundary-split}.

Using \(\diver u_\varepsilon=0\) in \(D_\varepsilon\) and
\(\diver_y\chi_j^\beta=-\delta_{j\beta}\) from \eqref{eq:unified-cell-momentum},
expand the product derivative and apply the cutoff conditions \eqref{eq:unified-cutoff-properties} to obtain
\[
\begin{aligned}
 b_\varepsilon(w_\varepsilon,h)
 ={}&\int_{D_\varepsilon}h\bigl[
 S_\varepsilon^2(\diver u_0)-\diver u_0
 -\varepsilon\eta_\varepsilon\chi_j^\beta(x/\varepsilon)\cdot
 S_\varepsilon^2(\nabla(\D u_0)_j^\beta)\bigr]\\
 &-\int_{D_\varepsilon\cap\Omega(c\varepsilon)}h\bigl[
 (1-\eta_\varepsilon)S_\varepsilon^2(\diver u_0)
 +\varepsilon\chi_j^\beta(x/\varepsilon)\cdot\nabla\eta_\varepsilon\,
 S_\varepsilon^2((\D u_0)_j^\beta)\bigr].
\end{aligned}
\]
Apply \eqref{eq:unified-smoothing-L2-bound}, \eqref{eq:unified-periodic-smoothing-product}
to the first integral and \eqref{eq:unified-periodic-smoothing-layer}, \eqref{eq:unified-cutoff-properties} to the second.
Together with the cell energy bound \eqref{eq:cell-pair-energy-bound}, this yields \eqref{eq:unified-G-bulk-boundary-split}.\qedhere
\end{proof}

By \eqref{eq:unified-F-bulk-boundary-split}--\eqref{eq:unified-G-bulk-boundary-split} and the stability argument in Proposition~\ref{thm:micro-full-pressure-infsup},
\begin{equation}\label{eq:unified-smoothed-joint-rate}
 \|w_\varepsilon\|_{H^1(\Omega)}
 +\|q_\varepsilon\|_{L^2(D_\varepsilon)}
 \le C\varepsilon^{1/2}\|u_0\|_{H^2(\Omega)}.
\end{equation}
These constants depend only on the dimension, material parameters and geometry, and are independent of \(\varepsilon,u_0\).

\subsection{Convergence rates and optimality}
\label{sec:unified-dual-rate}

\begin{proof}[Proof of Theorem~\ref{thm:main-quantitative}\textup{(i)}]
By \eqref{eq:unified-periodic-smoothing-product} and \eqref{eq:smoothed-displacement-corrector}, it suffices to prove \(\|w_\varepsilon\|_{L^2(\Omega)}\le C\varepsilon\|u_0\|_{H^2(\Omega)}\).
Following \cite[Section~4 and Remark~4.2]{ShenZhuge2017Elasticity} and \cite[Section~5]{Gu2018Neumann}, take \((v_\varepsilon,\pi_\varepsilon)\in V\times L^2(D_\varepsilon)\) and \(v_0\in V\) satisfying
\begin{equation}\label{eq:unified-dual-mixed-problem}
\begin{gathered}
 \mathfrak a_\varepsilon(v_\varepsilon,\varphi)+b_\varepsilon(\varphi,\pi_\varepsilon)
 =\int_\Omega(\widehat A\D v_0):\D\varphi
 =\int_\Omega w_\varepsilon\cdot\varphi
 \quad(\varphi\in V),\\
 \diver v_\varepsilon=0\quad\text{in }D_\varepsilon.
\end{gathered}
\end{equation}
Choose an auxiliary traction \(h\in\mathcal R|_{\partial\Omega}\) such that
\[
 \int_{\partial\Omega}h\cdot r=-\int_\Omega w_\varepsilon\cdot r
 \quad(r\in\mathcal R),\qquad
 \|h\|_{H^{1/2}(\partial\Omega)}\le C\|w_\varepsilon\|_{L^2(\Omega)}.
\]
The boundary orthogonality defining \(V\) and \(h\in\mathcal R|_{\partial\Omega}\) give
\[
 \int_{\partial\Omega}h\cdot\varphi\,\mathrm dS=0
 \qquad(\varphi\in V).
\]
Together with the compatibility condition for \(h\), \eqref{eq:unified-dual-mixed-problem} yields \((\widehat A\D v_0)n=h\) on \(\partial\Omega\).
The stability argument in Proposition~\ref{thm:micro-full-pressure-infsup} and \(H^2\) regularity for traction problems \cite[Theorems~4.18(ii) and~4.16]{McLean2000} give
\begin{equation}\label{eq:unified-dual-global}
 \|v_\varepsilon\|_{H^1(\Omega)}
 +\|\pi_\varepsilon\|_{L^2(D_\varepsilon)}
 +\|v_0\|_{H^2(\Omega)}
 \le C\|w_\varepsilon\|_{L^2(\Omega)}.
\end{equation}
Extend \(v_0\) to \(\R^d\), retaining its notation, with support in a fixed bounded set and \(\|v_0\|_{H^2(\R^d)}\le C\|v_0\|_{H^2(\Omega)}\).
Define the corrected errors for \((v_\varepsilon,\pi_\varepsilon,v_0)\) by \eqref{eq:smoothed-displacement-corrector}--\eqref{eq:smoothed-pressure-corrector}. The common right-hand side in \eqref{eq:unified-dual-mixed-problem} cancels upon subtraction, so Lemma~\ref{lem:unified-residual-bound} and \eqref{eq:unified-smoothed-joint-rate} apply.
By \eqref{eq:cell-pair-energy-bound} and \eqref{eq:unified-boundary-layer-estimate}--\eqref{eq:unified-cutoff-properties},
\[
\begin{aligned}
 &\|\nabla v_0\|_{L^2(\Omega(c\varepsilon))}
 +\left\|\nabla\!\left[
 \varepsilon\chi_j^\beta(x/\varepsilon)\eta_\varepsilon
 S_\varepsilon^2((\D v_0)_j^\beta)
 \right]\right\|_{L^2(\Omega(c\varepsilon))}\\
 &\quad+\left\|r_j^\beta(x/\varepsilon)\eta_\varepsilon
 S_\varepsilon^2((\D v_0)_j^\beta)
 \right\|_{L^2(D_\varepsilon\cap\Omega(c\varepsilon))}
 \le C\varepsilon^{1/2}\|v_0\|_{H^2(\Omega)}.
\end{aligned}
\]
Combining this with \eqref{eq:unified-smoothed-joint-rate} and \eqref{eq:unified-dual-global} gives
\begin{equation}\label{eq:unified-dual-layer}
 \|\nabla v_\varepsilon\|_{L^2(\Omega(c\varepsilon))}
 +\|\pi_\varepsilon\|_{L^2(D_\varepsilon\cap\Omega(c\varepsilon))}
 \le C\varepsilon^{1/2}\|v_0\|_{H^2(\Omega)}
 \le C\varepsilon^{1/2}\|w_\varepsilon\|_{L^2(\Omega)}.
\end{equation}
Take \(\varphi=w_\varepsilon\) in \eqref{eq:unified-dual-mixed-problem}, and test \eqref{eq:unified-F-bulk-boundary-split} and \eqref{eq:unified-G-bulk-boundary-split} with \(v_\varepsilon,\pi_\varepsilon\), respectively.
The symmetry of \(\mathfrak a_\varepsilon\), \(b_\varepsilon(v_\varepsilon,q_\varepsilon)=0\), and \eqref{eq:unified-dual-global}--\eqref{eq:unified-dual-layer} yield
\begin{equation}\label{eq:unified-mixed-duality}
 \|w_\varepsilon\|_{L^2(\Omega)}^2
 =\mathfrak a_\varepsilon(w_\varepsilon,v_\varepsilon)+b_\varepsilon(w_\varepsilon,\pi_\varepsilon)
 \le C\varepsilon\|u_0\|_{H^2(\Omega)}\|w_\varepsilon\|_{L^2(\Omega)}.
\end{equation}
Therefore
\begin{equation}\label{eq:unified-smoothed-displacement-L2-rate}
 \|w_\varepsilon\|_{L^2(\Omega)}
 \le C\varepsilon\|u_0\|_{H^2(\Omega)}.
\end{equation}
Together with \eqref{eq:unified-periodic-smoothing-product} and \eqref{eq:smoothed-displacement-corrector}, this proves \eqref{eq:main-L2-rate}.
\end{proof}

The following two lemmas will be used in the proof of Theorem~\ref{thm:main-quantitative}\textup{(ii)}; their proofs are given in Appendix~\ref{app:quantitative-estimates}.

\Needspace{4\baselineskip}
\begin{lemma}[Local energy estimate]
\label{lem:unified-local-energy}
Let \(B_R\subset\Omega\), \(G_+\in L^2(B_R\cap\Omega_\varepsilon;\R^{d\times d})\), \(G_-\in L^2(B_R\cap D_\varepsilon;\R^{d\times d})\), and \(f\in L^2(D_\varepsilon\cap B_R)\).
Suppose that \((u,p)\in H^1(B_R;\R^d)\times L^2(D_\varepsilon\cap B_R)\) satisfies, for every \(\varphi\in H_0^1(B_R;\R^d)\),
\begin{gather}
 \int_{B_R\cap\Omega_\varepsilon}\sigma(u):\D\varphi
 +\int_{B_R\cap D_\varepsilon}\sigma(u,p):\D\varphi
 =\int_{B_R\cap\Omega_\varepsilon}G_+:\nabla\varphi
  +\int_{B_R\cap D_\varepsilon}G_-:\nabla\varphi,\notag\\
 \diver u=f\quad\text{in }D_\varepsilon\cap B_R.
 \label{eq:unified-local-mixed-problem}
\end{gather}
If \(|B_t\setminus D_\varepsilon|/|B_t|\) is uniformly bounded below by a positive constant for all concentric balls \(B_t\), \(R/2\le t\le R\), then
\begin{equation}\label{eq:unified-local-pressure-control}
 \|p\|_{L^2(D_\varepsilon\cap B_t)}
 \le C\bigl(\|\D u\|_{L^2(B_t)}+\|G_+\|_{L^2(B_t\cap\Omega_\varepsilon)}+\|G_-\|_{L^2(B_t\cap D_\varepsilon)}\bigr),
\end{equation}
and
\begin{equation}\label{eq:unified-inhomogeneous-caccioppoli}
\begin{aligned}
 &\|\nabla u\|_{L^2(B_r)}+\|p\|_{L^2(D_\varepsilon\cap B_r)}
 \le C\left[\frac{\inf_{b\in\R^d}\|u-b\|_{L^2(B_R)}}{R-r}\right.\\
 &\qquad\left.
   +\|G_+\|_{L^2(B_R\cap\Omega_\varepsilon)}
   +\|G_-\|_{L^2(B_R\cap D_\varepsilon)}
   +\|f\|_{L^2(D_\varepsilon\cap B_R)}\right],
 \qquad R/2\le r<R.
\end{aligned}
\end{equation}
Here \(C\) depends only on \(d,\lambda,\mu,\widetilde\mu\) and this lower bound. Estimate \eqref{eq:unified-inhomogeneous-caccioppoli} also holds with \(\nabla u\) replaced by \(\D u\) and the infimum taken over rigid displacements.
\end{lemma}

The volume-fraction condition is the only geometric assumption in Lemma~\ref{lem:unified-local-energy}. Under \textup{(G)}, its lower bound can be taken as \(2^{-d}|Y\setminus\overline\omega|\) whenever \(R\ge4\sqrt d\,\varepsilon\), uniformly under cell translations.

\begin{lemma}[Higher integrability of the correctors]
\label{lem:unified-cell-higher-integrability}
Assume \textup{(G)}. The correctors in \eqref{eq:unified-cell-momentum} satisfy
\begin{equation}\label{eq:unified-cell-meyers}
 \|\chi_j^\beta\|_{W^{1,s}(Y)}
 +\|r_j^\beta\|_{L^s(\omega)}
 \le C,\qquad 1\le j,\beta\le d,
\end{equation}
for some \(s>2\) and \(C>0\) depending only on \(d,\lambda,\mu,\widetilde\mu\) and the Lipschitz geometry of \(\omega\).
\end{lemma}

\begin{proof}[Proof of Theorem~\ref{thm:main-quantitative}\textup{(ii)}]
We first prove \eqref{eq:main-raw-H2-joint-rate}. Choose \(s>2\) satisfying \eqref{eq:unified-cell-meyers}. When \(d=2\), take \(q=2\); for \(d\ge3\), also require
\begin{equation}\label{eq:unified-weaker-exponent}
 2<s<\min\{d,2d/(d-2)\},\qquad
 \frac1q=\frac12+\frac1d-\frac1s,\qquad 2<q<d.
\end{equation}
For periodic \(h\in L^s(Y)\) and \(\psi\in W^{1,q}(\R^d)\) supported in a fixed bounded set, the cellwise H\"older--Sobolev inequality and summation give
\begin{equation}\label{eq:unified-critical-periodic-multiplier}
 \|h(x/\varepsilon)\psi\|_{L^2(\R^d)}
 \le C\|h\|_{L^s(Y)}
 \left(\|\psi\|_{L^2(\R^d)}
       +\varepsilon\|\nabla\psi\|_{L^q(\R^d)}\right).
\end{equation}
By \eqref{eq:unified-cell-meyers} and Sobolev embedding, \(\chi_j^\beta\in L^{ds/(d-s)}(Y)\) for \(d\ge3\), or \(L^\infty(Y)\) for \(d=2\).
Combining \eqref{eq:unified-critical-periodic-multiplier}, \eqref{eq:unified-smoothing-L2-bound}, and the boundedness of \(S_\varepsilon^2\) on \(L^q\), we obtain
\begin{equation}\label{eq:unified-bulk-desmoothing}
 \|\varepsilon\chi_j^\beta(x/\varepsilon)
 [\D u_0-S_\varepsilon^2(\D u_0)]_j^\beta\|_{H^1(\Omega)}
 +\|r_j^\beta(x/\varepsilon)
 [\D u_0-S_\varepsilon^2(\D u_0)]_j^\beta\|_{L^2(D_\varepsilon)}
 \le C\varepsilon\|u_0\|_{W^{2,q}(\Omega)}.
\end{equation}
By \eqref{eq:unified-periodic-smoothing-product}--\eqref{eq:unified-cutoff-properties} and \eqref{eq:cell-pair-energy-bound}, the cutoff error satisfies
\[
 \left\|\varepsilon\chi_j^\beta(x/\varepsilon)(1-\eta_\varepsilon)
 S_\varepsilon^2((\D u_0)_j^\beta)\right\|_{H^1(\Omega)}
 +\left\|r_j^\beta(x/\varepsilon)(1-\eta_\varepsilon)
 S_\varepsilon^2((\D u_0)_j^\beta)\right\|_{L^2(D_\varepsilon)}
 \le C\varepsilon^{1/2}\|u_0\|_{H^2(\Omega)}.
\]
Substitution into \eqref{eq:smoothed-displacement-corrector}--\eqref{eq:smoothed-pressure-corrector}, together with \eqref{eq:unified-smoothed-joint-rate}, \eqref{eq:unified-bulk-desmoothing}, and \(q\ge2\), gives \eqref{eq:main-raw-H2-joint-rate}.
If \eqref{eq:unified-bounded-correctors} holds, then \(\chi,\nabla_y\chi,r\) are bounded, so the estimates apply with \(q=2\); in particular, this holds when \(\partial\omega\in C^{1,\kappa}\).

We next prove \eqref{eq:main-interior-joint-rate}. Assume \(\partial\Omega\in C^{1,1}\) and
fix \(\Omega'\Subset\Omega\). By \eqref{eq:main-raw-H2-joint-rate}, it suffices to consider sufficiently small \(\varepsilon\). We first prove
\begin{equation}\label{eq:unified-interior-smoothed-joint-rate}
 \|w_\varepsilon\|_{H^1(\Omega')}
 +\|q_\varepsilon\|_{L^2(D_\varepsilon\cap\Omega')}
 \le C_{\Omega'}\varepsilon\|u_0\|_{H^2(\Omega)}.
\end{equation}
Choose an \(\varepsilon\)-independent finite covering \(\Omega'\subset\bigcup_{k=1}^N B_R(x_k)\), with \(B_{2R}(x_k)\Subset\Omega\). For sufficiently small \(\varepsilon\),
\[
 \eta_\varepsilon=1\quad\text{on }B_{2R}(x_k),\qquad
 B_{2R}(x_k)\cap K_\varepsilon=\varnothing,\qquad
 2R\ge4\sqrt d\,\varepsilon.
\]
In \eqref{eq:unified-momentum-residual-expansion} and the divergence expansion in Lemma~\ref{lem:unified-residual-bound}, the cutoff and buffer-region terms vanish.
Since \(A\) is bounded and the smoothing kernel is nonnegative, applying \eqref{eq:unified-inhomogeneous-caccioppoli} on \(B_{2R}(x_k)\) with \(b=0\) gives
\begin{equation}\label{eq:unified-local-residual-densities}
\begin{aligned}
 \|\nabla w_\varepsilon\|_{L^2(B_R(x_k))}
 +\|q_\varepsilon\|_{L^2(D_\varepsilon\cap B_R(x_k))}
 &\le \frac{C}{R}\|w_\varepsilon\|_{L^2(B_{2R}(x_k))}
 +C\|\D u_0-S_\varepsilon^2(\D u_0)\|_{L^2(B_{2R}(x_k))}\\
 &\quad+C\varepsilon\|(|\phi|+|\chi|)(x/\varepsilon)
 S_\varepsilon^2(|\nabla\D u_0|)\|_{L^2(B_{2R}(x_k))}.
\end{aligned}
\end{equation}
Estimate the three terms on the right of \eqref{eq:unified-local-residual-densities} using \eqref{eq:unified-smoothed-displacement-L2-rate}, \eqref{eq:unified-smoothing-L2-bound}, and \eqref{eq:unified-periodic-smoothing-product} together with \eqref{eq:unified-flux-potential}, \eqref{eq:cell-pair-energy-bound}, respectively.
Squaring and summing over the covering balls, and using \eqref{eq:unified-smoothed-displacement-L2-rate} to control \(\|w_\varepsilon\|_{L^2(\Omega')}\), gives \eqref{eq:unified-interior-smoothed-joint-rate}.
Since \(\eta_\varepsilon=1\) on \(\Omega'\), \eqref{eq:smoothed-displacement-corrector}--\eqref{eq:smoothed-pressure-corrector} and \eqref{eq:unified-bulk-desmoothing} yield \eqref{eq:main-interior-joint-rate}.

\end{proof}

\noindent\begin{minipage}{\linewidth}
The convergence order in Theorem~\ref{thm:main-quantitative}\textup{(i)} is sharp in general.

\begin{proposition}[Optimal \texorpdfstring{\(L^2\)}{L2} convergence order for the displacement]
\label{prop:unified-L2-sharpness}
Assume \textup{(G)} and \(\partial\Omega\in C^{1,1}\).
There is an \(\varepsilon\)-independent traction \(g\in H_{\mathcal R}^{-1/2}(\partial\Omega;\R^d)\)
such that the corresponding displacement solutions \(u_\varepsilon,u_0\) of \eqref{eq:unified-mixed-problem} and \eqref{eq:unified-effective-problem} satisfy
\begin{equation}\label{eq:unified-L2-sharpness}
 c\varepsilon\le\|u_\varepsilon-u_0\|_{L^2(\Omega)}
 \le C\varepsilon
 \qquad(0<\varepsilon<\varepsilon_0),
\end{equation}
for some \(c,C,\varepsilon_0>0\), all independent of \(\varepsilon\).
\end{proposition}
\end{minipage}

\begin{proof}
Take \(u_0(x)=x-\dashint_{\partial\Omega}y\,\mathrm dS_y\in V\) and \(g=(\widehat A\mathbb I)n\). Since \(\widehat A\mathbb I\) is a constant symmetric stress, \(g\in H_{\mathcal R}^{-1/2}(\partial\Omega;\R^d)\), and \(u_0\) solves the effective problem \eqref{eq:unified-effective-problem} with \(\D u_0=\mathbb I\) and \(\diver u_0=d\). Estimate \eqref{eq:main-L2-rate} gives the upper bound.

Take a nonzero nonnegative \(\psi\in C_c^\infty(\omega)\), extend it by zero to \(Y\), and then periodically. By \eqref{eq:unified-cutoff-properties}, \(\eta_\varepsilon\psi(x/\varepsilon)\in C_c^\infty(D_\varepsilon)\). Periodic averaging gives
\[
 \int_\Omega\eta_\varepsilon\psi(x/\varepsilon)\,\mathrm dx
 \longrightarrow |\Omega|\int_\omega\psi>0.
\]
Since \(\diver(u_\varepsilon-u_0)=-d\) in the inclusions, integration by parts and \eqref{eq:unified-cutoff-properties} yield, for sufficiently small \(\varepsilon\),
\begingroup
\postdisplaypenalty=10000
\[
\begin{aligned}
 0<c&\le d\int_\Omega\eta_\varepsilon\psi(x/\varepsilon)\,\mathrm dx
 =\int_\Omega(u_\varepsilon-u_0)\cdot
       \nabla[\eta_\varepsilon\psi(x/\varepsilon)]\,\mathrm dx\\
 &\le C\varepsilon^{-1}\|u_\varepsilon-u_0\|_{L^2(\Omega)}.
\end{aligned}
\]
This proves the lower bound in \eqref{eq:unified-L2-sharpness}.
\endgroup
\end{proof}

\section{Uniform Lipschitz and Full Pressure Estimates}
\label{sec:uniform-lipschitz-pressure}

To prove Theorem~\ref{thm:positive-boundary-pressure}, we first establish uniform interior and boundary Lipschitz estimates, then control the full pressure by the displacement gradient.

\subsection{Interior Lipschitz estimates}
\label{subsec:il-local-solutions}
\label{sec:unified-interior-lipschitz}

To account for local rescaling about arbitrary centers, introduce a shift \(\xi\in\R^d/\mathbb Z^d\) and set
\begin{equation}\label{eq:shifted-phases}
 D_\varepsilon^\xi
 =\bigcup_{k\in\mathbb Z^d}\varepsilon(k+\omega-\xi).
\end{equation}

Let \(B\subset\R^d\) be a ball, and suppose that
\((u_\varepsilon,p_\varepsilon)\in H^1(B;\R^d)\times L^2(B\cap D_\varepsilon^\xi)\)
satisfies, for every \(\varphi\in H_0^1(B;\R^d)\),
\begin{gather}
 \int_{B\setminus\overline{D_\varepsilon^\xi}}\sigma(u_\varepsilon):\D\varphi
 +\int_{B\cap D_\varepsilon^\xi}\sigma(u_\varepsilon,p_\varepsilon):\D\varphi=0,\notag\\
 \diver u_\varepsilon=0\quad\text{in }B\cap D_\varepsilon^\xi.
 \label{eq:il-homogeneous-system}
\end{gather}
Under \textup{(G)}, let \((u_\varepsilon,p_\varepsilon)\) satisfy \eqref{eq:il-homogeneous-system} in \(B_R(x_0)\).
There exist constants \(C_0\ge4\sqrt d\) and \(C>0\) such that, whenever \(0<r<R\) and \(R-r\ge C_0\varepsilon\),
\begin{equation}\label{eq:il-caccioppoli}
 \int_{B_r(x_0)}|\nabla u_\varepsilon|^2
 \le \frac{C}{(R-r)^2}\inf_{b\in\R^d}
 \int_{B_R(x_0)}|u_\varepsilon-b|^2.
\end{equation}
Here \(C_0,C\) are independent of \(\varepsilon,r,R,x_0,\xi\).
Since \(R\ge R-r\ge C_0\varepsilon\ge4\sqrt d\,\varepsilon\), the volume fraction bound following Lemma~\ref{lem:unified-local-energy} applies. Taking \(G_\pm=f=0\) in \eqref{eq:unified-inhomogeneous-caccioppoli}, dropping the pressure term, and squaring gives \eqref{eq:il-caccioppoli} for \(r\ge R/2\). For \(r<R/2\), use inner radius \(R/2\), together with \(B_r(x_0)\subset B_{R/2}(x_0)\) and \(R-r\le R\).

To identify the limiting stress, define the corrected solution for an affine function \(\ell(x)=a+M(x-x_0)\) by
\begin{equation}\label{eq:il-corrected-affine}
 \ell^\varepsilon(x)
 =\ell(x)+\varepsilon M_j^\beta\chi_j^\beta(x/\varepsilon+\xi).
\end{equation}
By the cell equations \eqref{eq:unified-cell-momentum}, the pair
\((\ell^\varepsilon,M_j^\beta r_j^\beta(x/\varepsilon+\xi))\)
satisfies \eqref{eq:il-homogeneous-system} in every ball \(B\subset\R^d\).

\begin{lemma}[Local homogenization approximation]
\label{thm:il-compactness}
Assume \textup{(G)}.
Let \((u_\varepsilon,p_\varepsilon)\in H^1(B_1;\R^d)\times L^2(B_1\cap D_\varepsilon^\xi)\)
satisfy \eqref{eq:il-homogeneous-system} in \(B_1\), with \(\left(\dashint_{B_1}|u_\varepsilon|^2\right)^{1/2}\le1\).
For every prescribed \(\delta>0\), there exists \(\varepsilon_0>0\) such that, for all
\(\xi\in\R^d/\mathbb Z^d\) and \(0<\varepsilon<\varepsilon_0\),
there exists \(u_0\in H^1(B_{1/2};\R^d)\) satisfying
\begin{equation}\label{eq:il-compactness-approx}
 \begin{gathered}
 \diver(\widehat A\nabla u_0)=0\quad\text{in }B_{1/2},\\
 \left(\dashint_{B_{1/2}}|u_\varepsilon-u_0|^2\right)^{1/2}\le\delta,
 \qquad
 \left(\dashint_{B_{1/2}}|u_0|^2\right)^{1/2}\le C.
 \end{gathered}
\end{equation}
Here \(\widehat A\) is defined by \eqref{eq:effective-tensor-inclusions}, \(C\) is independent of \(\varepsilon,\delta,\xi\), and \(\varepsilon_0=\varepsilon_0(\delta)>0\) is independent of \(\xi\).
\end{lemma}

\begin{proof}
Suppose otherwise. There exist \(\delta>0\), \(\varepsilon_j\downarrow0\), shifts \(\xi_j\), and solutions \((u_j,p_j)\) satisfying the equations and displacement bound in the lemma, but no \(u_0\) satisfies \eqref{eq:il-compactness-approx} for any \(j\).

Extend \(p_j\) by zero to \(B_{1/2}\). By \eqref{eq:il-caccioppoli} and \eqref{eq:unified-local-pressure-control}, for sufficiently large \(j\),
\[
 \|u_j\|_{H^1(B_{1/2};\R^d)}+\|p_j\|_{L^2(B_{1/2})}\le C.
\]
Define the full stress \(\sigma_j\) by
\[
 (\sigma_j)_i^\alpha
 =A_{ik}^{\alpha\beta}(x/\varepsilon_j+\xi_j)\partial_{x_\beta}u_j^k
   +p_j\delta_{i\alpha}.
\]
Passing to a subsequence, we have
\begin{equation}\label{eq:il-solution-convergence}
 \begin{gathered}
 u_j\rightharpoonup u_0\text{ in }H^1(B_{1/2};\R^d),\qquad
 u_j\to u_0\text{ in }L^2(B_{1/2}),\\
 \sigma_j\rightharpoonup\sigma\text{ in }L^2(B_{1/2};\R^{d\times d}).
 \end{gathered}
\end{equation}
By \eqref{eq:il-homogeneous-system}, \(\diver\sigma_j=0\), and hence \(\diver\sigma=0\).

To identify \(\sigma\), fix any constant matrix \(M=(M_k^\beta)\) and, following \eqref{eq:il-corrected-affine}, set
\[
 \ell^{\varepsilon_j}(x)
 =Mx+\varepsilon_jM_k^\beta\chi_k^\beta(x/\varepsilon_j+\xi_j).
\]
Then
\begin{equation}\label{eq:il-corrector-convergence}
 \begin{gathered}
 \ell^{\varepsilon_j}\to Mx\text{ in }L^2(B_{1/2};\R^d),\\
 \bigl[\widehat A_{ik}^{\alpha\beta}+\mathcal B_{ik}^{\alpha\beta}(x/\varepsilon_j+\xi_j)\bigr]M_k^\beta
 \rightharpoonup\widehat A_{ik}^{\alpha\beta}M_k^\beta\text{ in }L^2(B_{1/2})
 \end{gathered}
 \qquad(\forall\,\xi_j\in\R^d/\mathbb Z^d).
\end{equation}

For \(\eta\in C_c^\infty(B_{1/2})\), test \eqref{eq:il-homogeneous-system} for \((u_j,p_j)\) by \(\eta\ell^{\varepsilon_j}\), and the weak momentum equation for the pair associated with \eqref{eq:il-corrected-affine} by \(\eta u_j\). Subtracting and using \eqref{eq:unified-flux-discrepancy}, the symmetry of \(A\), and \(\diver u_j=\diver\ell^{\varepsilon_j}=0\) in the inclusions gives
\begin{equation}\label{eq:il-reciprocity}
 \int_{B_{1/2}}(\sigma_j)_i^\alpha(\ell^{\varepsilon_j})^i\partial_{x_\alpha}\eta
 =\int_{B_{1/2}}
 \bigl[\widehat A_{ik}^{\alpha\beta}+\mathcal B_{ik}^{\alpha\beta}(x/\varepsilon_j+\xi_j)\bigr]
 M_k^\beta u_j^i\partial_{x_\alpha}\eta.
\end{equation}
The strong and weak convergences in \eqref{eq:il-solution-convergence} and \eqref{eq:il-corrector-convergence} allow us to pass to the limit in \eqref{eq:il-reciprocity} and integrate by parts.
Using \(\diver\sigma=0\) and the effective tensor symmetry \eqref{eq:effective-tensor-symmetry}, we obtain
\begin{equation}\label{eq:il-stress-identification}
 \int_{B_{1/2}}
 \bigl(\sigma_i^\alpha-\widehat A_{ik}^{\alpha\beta}\partial_{x_\beta}u_0^k\bigr)
 M_i^\alpha\eta=0.
\end{equation}
By \eqref{eq:il-stress-identification} and the arbitrariness of \(M,\eta\), \(\sigma_i^\alpha=\widehat A_{ik}^{\alpha\beta}\partial_{x_\beta}u_0^k\).
Thus \(u_0\) satisfies the effective equation in \eqref{eq:il-compactness-approx}.

By the strong convergence in \eqref{eq:il-solution-convergence} and \(\dashint_{B_1}|u_j|^2\le1\), for sufficiently large \(j\), \(u_0\) satisfies \eqref{eq:il-compactness-approx} with \(C=2^{d/2}\), a contradiction.
\end{proof}

\begin{lemma}[Large-scale interior Lipschitz estimate]
\label{lem:il-large-scale-gradient}
Assume \textup{(G)} holds and \(\partial\omega\in C^{1,\kappa}\), where \(0<\kappa<1\).
Let \((u_\varepsilon,p_\varepsilon)\in H^1(B_R(x_0);\R^d)\times L^2(B_R(x_0)\cap D_\varepsilon^\xi)\) satisfy \eqref{eq:il-homogeneous-system} in \(B_R(x_0)\).
There exist constants \(C_0\ge4\sqrt d\) and \(C>0\) such that, for every \(C_0\varepsilon\le r\le R/2\),
\begin{equation}\label{eq:il-large-scale-gradient}
 \left(\dashint_{B_r(x_0)}|\nabla u_\varepsilon|^2\right)^{1/2}
 \le \frac{C}{R}\inf_{b\in\R^d}
 \left(\dashint_{B_R(x_0)}|u_\varepsilon-b|^2\right)^{1/2}.
\end{equation}
Here \(C_0,C\) are independent of \(\varepsilon,r,R,x_0,\xi\).
\end{lemma}

\begin{proof}
Set
\begin{equation}\label{eq:il-excess}
 H(r)=\frac1r
 \inf_{\substack{\ell(x)=a+M(x-x_0)\\
 a\in\R^d,\;M\in\R^{d\times d}}}
 \left(\dashint_{B_r(x_0)}|u_\varepsilon-\ell^\varepsilon|^2\right)^{1/2},
\end{equation}
where \(\ell^\varepsilon\) is given by \eqref{eq:il-corrected-affine}, and the infimum is attained by finite dimensionality. For \(0<r\le R\), set
\[
 \begin{gathered}
 \widetilde u(z)=r^{-1}u_\varepsilon(x_0+rz),\qquad
 \widetilde p(z)=p_\varepsilon(x_0+rz),\\
 \widetilde\varepsilon=\varepsilon/r,\qquad
 \widetilde\xi=\xi+x_0/\varepsilon,\qquad
 \widetilde H(t)=H(rt).
 \end{gathered}
\]
This transformation preserves \eqref{eq:il-homogeneous-system}, so we first work on the unit ball.
If \(H(1)=0\), then \(H(t)=0\) for \(0<t\le1\); otherwise, choose a minimizer \(\ell(x)=a+Mx\) and normalize by
\[
 (u_\varepsilon,p_\varepsilon)\longmapsto
 H(1)^{-1}\bigl(u_\varepsilon-\ell^\varepsilon,\,
 p_\varepsilon-M_j^\beta r_j^\beta(x/\varepsilon+\xi)\bigr).
\]
Retaining the original notation, we may assume \(H(1)=\left(\dashint_{B_1}|u_\varepsilon|^2\right)^{1/2}=1\).

Take \(0<\theta<1/4\) and \(\delta>0\). When \(\varepsilon<\varepsilon_0(\delta)\), by Lemma~\ref{thm:il-compactness}, there exists \(u_0\) satisfying \eqref{eq:il-compactness-approx}. Set
\[
 \ell(x)=u_0(0)+\nabla u_0(0)x.
\]
In the definition of \(H(\theta)\) in \eqref{eq:il-excess}, use \(\ell^\varepsilon\) as a competitor. The triangle inequality gives
\begin{equation}\label{eq:il-taylor-comparison}
 \begin{aligned}
 H(\theta)
 &\le \frac1\theta\left(\dashint_{B_\theta}|u_\varepsilon-u_0|^2\right)^{1/2}
       +\frac1\theta\|u_0-\ell\|_{L^\infty(B_\theta)}
       +\frac1\theta\|\ell-\ell^\varepsilon\|_{L^\infty(B_\theta)}\\
 &\le C\theta^{-1-d/2}\delta+C\theta+C\varepsilon/\theta.
 \end{aligned}
\end{equation}
The second term follows from standard interior estimates for the constant-coefficient elliptic system \(\diver(\widehat A\nabla u_0)=0\) in \eqref{eq:il-compactness-approx} and Taylor's formula.

Choose \(\theta,\delta,C_0\), in this order, so that
\[
 \begin{gathered}
 C\theta\le\tfrac16,\qquad C\theta^{-1-d/2}\delta\le\tfrac16,\\
 C_0^{-1}<\varepsilon_0(\delta),\qquad \frac{C}{\theta C_0}\le\tfrac16.
 \end{gathered}
\]
For \(\varepsilon\le C_0^{-1}\), \eqref{eq:il-taylor-comparison} gives \(H(\theta)\le1/2\) under the above normalization. Undoing the normalization gives \(H(\theta)\le\tfrac12H(1)\). The scaling relations \(\widetilde H(t)=H(rt)\) and \(\widetilde\varepsilon=\varepsilon/r\) then yield
\begin{equation}\label{eq:il-one-step}
 H(\theta r)\le\frac12H(r),\qquad C_0\varepsilon\le r\le R.
\end{equation}

\begin{samepage}
The integral formula for affine functions on a ball, together with \eqref{eq:il-corrected-affine}, \eqref{eq:unified-bounded-correctors}, and the triangle inequality, gives
\[
 r|\nabla\ell|
 =\sqrt{d+2}\inf_{b\in\R^d}
   \left(\dashint_{B_r(x_0)}|\ell-b|^2\right)^{1/2}
 \le C\inf_{b\in\R^d}
   \left(\dashint_{B_r(x_0)}|\ell^\varepsilon-b|^2\right)^{1/2}
   +C\varepsilon|\nabla\ell|.
\]
Choose \(C_0\ge2C\) and move the last term to the left to obtain
\begin{equation}\label{eq:il-affine-lower}
 r|\nabla\ell|\le C\inf_{b\in\R^d}
 \left(\dashint_{B_r(x_0)}|\ell^\varepsilon-b|^2\right)^{1/2},
 \qquad r\ge C_0\varepsilon.
\end{equation}
\end{samepage}
Set \(r_k=\theta^kR\). Whenever \(r_k\ge C_0\varepsilon\), choose an affine minimizer for \(H(r_k)\) and denote it by \(\ell_k\).

When \(r_{k+1}\ge C_0\varepsilon\), apply \eqref{eq:il-affine-lower} to \(\ell_{k+1}-\ell_k\). The minimizing property, the triangle inequality, and \eqref{eq:il-one-step} give
\begin{equation}\label{eq:il-affine-increments}
 |\nabla\ell_{k+1}-\nabla\ell_k|
 \le C\bigl[H(r_{k+1})+\theta^{-1-d/2}H(r_k)\bigr]
 \le C H(r_k).
\end{equation}
By \eqref{eq:il-one-step}, \(H(r_k)\le2^{-k}H(R)\). Summing \eqref{eq:il-affine-increments}, and then bounding the initial quantities using \eqref{eq:il-affine-lower}, the triangle inequality, and a constant competitor in \eqref{eq:il-excess}, gives
\begin{equation}\label{eq:il-iteration-control}
 H(r_k)+|\nabla\ell_k|
 \le\frac C R\inf_{b\in\R^d}
 \left(\dashint_{B_R(x_0)}|u_\varepsilon-b|^2\right)^{1/2},
 \qquad r_k\ge C_0\varepsilon.
\end{equation}
For \(k\ge1\) with \(r_k\ge C_0\varepsilon\), since \(\theta<1/4\), we have \(r_{k-1}-r_k=(\theta^{-1}-1)r_k\ge C_0\varepsilon\). By \eqref{eq:il-corrected-affine} and linearity, the pair
\[
 \bigl(u_\varepsilon-\ell_{k-1}^{\varepsilon},\,
 p_\varepsilon-\partial_j\ell_{k-1}^{\beta}r_j^\beta(x/\varepsilon+\xi)\bigr)
\]
still satisfies \eqref{eq:il-homogeneous-system} in \(B_{r_{k-1}}(x_0)\). Applying \eqref{eq:il-caccioppoli} on the inner ball \(B_{r_k}(x_0)\) gives
\[
 \left(\dashint_{B_{r_k}(x_0)}
 |\nabla(u_\varepsilon-\ell_{k-1}^{\varepsilon})|^2\right)^{1/2}
 \le\frac{C}{r_{k-1}}
 \left(\dashint_{B_{r_{k-1}}(x_0)}
 |u_\varepsilon-\ell_{k-1}^{\varepsilon}|^2\right)^{1/2}
 =CH(r_{k-1}).
\]
By \eqref{eq:unified-bounded-correctors}, \(\|\nabla\ell_{k-1}^{\varepsilon}\|_{L^\infty}\le C|\nabla\ell_{k-1}|\). Combining the triangle inequality with \eqref{eq:il-iteration-control} gives
\begin{equation}\label{eq:il-iteration-gradient}
 \left(\dashint_{B_{r_k}(x_0)}|\nabla u_\varepsilon|^2\right)^{1/2}
 \le C\bigl[H(r_{k-1})+|\nabla\ell_{k-1}|\bigr]
 \le\frac C R\inf_{b\in\R^d}
 \left(\dashint_{B_R(x_0)}|u_\varepsilon-b|^2\right)^{1/2}.
\end{equation}
Let \(C_0\varepsilon\le r\le R/2\). If \(r<\theta R\), choose \(r\le r_k<r/\theta\); volume comparison and \eqref{eq:il-iteration-gradient} give \eqref{eq:il-large-scale-gradient}. If \(r\ge\theta R\), then \(R-r\ge R/2\ge C_0\varepsilon\) and \(R/r\le\theta^{-1}\), so \eqref{eq:il-caccioppoli} applies directly.
\end{proof}

\begin{samepage}
\begin{lemma}[Small-scale gradient estimate]
\label{lem:il-small-scale-gradient}
Assume \textup{(G)} and \(\partial\omega\in C^{1,\kappa}\), where \(0<\kappa<1\).
Fix \(L\ge1\), let \(0<\rho\le L\varepsilon\), and suppose that \((u_\varepsilon,p_\varepsilon)\) satisfies \eqref{eq:il-homogeneous-system} in \(B_{2\rho}(y)\). Then
\begin{equation}\label{eq:il-fixed-gradient}
 \|\nabla u_\varepsilon\|_{L^\infty(B_{\rho/2}(y))}
 \le C_1\left(\dashint_{B_\rho(y)}|\nabla u_\varepsilon|^2\right)^{1/2}
 \le \frac{C_2}{\rho}\inf_{b\in\R^d}
 \left(\dashint_{B_{2\rho}(y)}|u_\varepsilon-b|^2\right)^{1/2}.
\end{equation}
Here \(C_1,C_2\) are independent of \(\varepsilon,\rho,y,\xi\). For the actual inclusion region \(D_\varepsilon\), replace \(D_\varepsilon^\xi\) by \(D_\varepsilon\) in \eqref{eq:il-homogeneous-system} and assume that \(B_{2\rho}(y)\Subset\Omega\). Then the same estimate holds.
\end{lemma}
\end{samepage}

Apply the Lam\'e and Stokes interior estimates (\cite[Theorems~5.19--5.20]{GiaquintaMartinazzi2012}, \cite[Theorem~4.4]{NguyenStara2007}) within each phase, and Lemma~\ref{lem:c1a-joint-caccioppoli} and Proposition~\ref{prop:interface-chart-estimate} near interfaces. Together with the phasewise Caccioppoli inequalities (see \cite[Theorem~3.1]{GuShen2015} for Stokes), scaling and covering give \eqref{eq:il-fixed-gradient}.

\begin{samepage}
\begin{proposition}[Uniform interior Lipschitz estimate]
\label{thm:unified-interior-lipschitz}
\label{thm:main-interior-lipschitz}
Assume \textup{(G)} and \(\partial\omega\in C^{1,\kappa}\), where \(0<\kappa<1\).
Let \(0<\varepsilon<1\) and \(R>0\), and let \((u_\varepsilon,p_\varepsilon)\) satisfy \eqref{eq:il-homogeneous-system} in \(B_{2R}(x_0)\). Then
\begin{equation}\label{eq:unified-interior-lipschitz}
 \|\nabla u_\varepsilon\|_{L^\infty(B_{R/2}(x_0))}
 \le \frac{C}{R}\inf_{b\in\R^d}
 \left(\dashint_{B_{2R}(x_0)}|u_\varepsilon-b|^2\right)^{1/2}.
\end{equation}
Here \(C\) is independent of \(\varepsilon,R,x_0,\xi\). For the actual inclusion region \(D_\varepsilon\), replace \(D_\varepsilon^\xi\) by \(D_\varepsilon\) in \eqref{eq:il-homogeneous-system} and assume that \(B_{2R}(x_0)\Subset\Omega\). Then the same estimate holds.
\end{proposition}
\end{samepage}

\begin{proof}
Take the constant \(C_0\) from Lemma~\ref{lem:il-large-scale-gradient}. If \(R<4C_0\varepsilon\), apply Lemma~\ref{lem:il-small-scale-gradient} with \(\rho=R\) and \(L=4C_0\) to obtain the conclusion. Henceforth assume that \(R\ge4C_0\varepsilon\).

Fix any \(y\in B_{R/2}(x_0)\). If the inclusion region is \(D_\varepsilon\), then every cell meeting \(B_R(y)\), denoted by \(Y_\varepsilon^k\), satisfies
\[
 \overline{Y_\varepsilon^k}
 \subset B_{R+\sqrt d\,\varepsilon}(y)
 \subset B_{2R}(x_0)\Subset\Omega.
\]
The corresponding inclusions \(\varepsilon(k+\omega)\) are all contained in \(D_\varepsilon\). Thus \(D_\varepsilon\cap B_R(y)=D_\varepsilon^0\cap B_R(y)\), where \(D_\varepsilon^0\) is the periodic inclusion region obtained by taking \(\xi=0\) in \eqref{eq:shifted-phases}.

Set \(\rho=C_0\varepsilon\le R/4\). First apply Lemma~\ref{lem:il-small-scale-gradient} with \(L=C_0\), then, on \(B_R(y)\), apply Lemma~\ref{lem:il-large-scale-gradient}; for the actual region \(D_\varepsilon\), take \(\xi=0\). Using \(B_R(y)\subset B_{2R}(x_0)\), we obtain, for a.e. \(y\in B_{R/2}(x_0)\),
\[
 |\nabla u_\varepsilon(y)|
 \le C\left(\dashint_{B_\rho(y)}|\nabla u_\varepsilon|^2\right)^{1/2}
 \le\frac C R\inf_{b\in\R^d}
 \left(\dashint_{B_{2R}(x_0)}|u_\varepsilon-b|^2\right)^{1/2}.
\]
Taking the essential supremum gives \eqref{eq:unified-interior-lipschitz}.
\end{proof}

\subsection{Boundary Lipschitz estimates}

\label{pos:pressure-energy}
\label{sec:positive-boundary-pressure}
\label{pos:iteration-subsection}
\label{pos:all-periods-subsection}
\label{pos:pressure-pointwise}
\label{pos:comparison-subsection}

Assume \(\partial\Omega,\partial\omega\in C^{1,\kappa}\) and \(0<\tau<\kappa<1\).
Choose a sufficiently small \(0<R_0\le1\), uniformly in \(x_0\in\partial\Omega\), so that in orthogonal coordinates \(y=(y',y_d)\) centered at \(x_0\), \(\Omega\cap B_{R_0}(x_0)\) is represented by
\begin{equation}\label{pos:boundary-chart}
 \{y\in B_{R_0}(0):y_d>\psi(y')\}.
\end{equation}
Here \(\psi(0)=0\), with \(\|\psi\|_{C^{1,\kappa}}\) bounded uniformly in \(x_0\).
In the original coordinates, for \(x\in\partial\Omega\) and \(0<r\le R_0\), set
\begin{equation}\label{pos:local-domains}
 U_r(x)=\Omega\cap B_r(x),\qquad
 \Delta_r(x)=\partial\Omega\cap B_r(x).
\end{equation}

Given \(x_0\in\partial\Omega\), \(0<R\le R_0\), and \(g\in L^2(\Delta_R(x_0);\R^d)\), let
\((u_\varepsilon,p_\varepsilon)\in H^1(U_R(x_0);\R^d)\times L^2(D_\varepsilon\cap U_R(x_0))\).
Assume that the following local weak formulation holds for every \(\varphi\in H^1(U_R(x_0);\R^d)\) with zero trace on \(\Omega\cap\partial B_R(x_0)\):

\begingroup
\medmuskip=1.5mu\relax
\thickmuskip=2mu\relax
\begin{gather}
 \int_{U_R(x_0)\cap\Omega_\varepsilon}[\lambda\diver u_\varepsilon\,\diver\varphi+2\mu\D u_\varepsilon:\D\varphi]
 +\int_{U_R(x_0)\cap D_\varepsilon}[2\widetilde\mu\D u_\varepsilon:\D\varphi+p_\varepsilon\diver\varphi]
 =\int_{\Delta_R(x_0)}g\cdot\varphi,\notag\\
 \diver u_\varepsilon=0\quad\text{in }D_\varepsilon\cap U_R(x_0).
 \label{pos:local-momentum}
\end{gather}
\endgroup

\begin{samepage}
\begin{lemma}[Boundary local energy estimate]
\label{lem:boundary-local-energy}
Under the assumptions of this subsection, let \(x_0\in\partial\Omega\), \(0<R\le R_0\), and suppose that \((u_\varepsilon,p_\varepsilon)\) satisfies the local weak formulation \eqref{pos:local-momentum} in \(U_R(x_0)\), with \(g\in L^2(\Delta_R(x_0);\R^d)\).
Then, for every \(0<r\le R\), \(0<\rho<1\), and \(b\in\R^d\),
\begin{equation}\label{pos:boundary-energy-bound}
 \begin{aligned}
 \|\nabla u_\varepsilon\|_{L^2(U_{\rho r}(x_0))}
 +\|p_\varepsilon\|_{L^2(U_{\rho r}(x_0)\cap D_\varepsilon)}
 \le C\left(\frac{\|u_\varepsilon-b\|_{L^2(U_r(x_0))}}r
 +r^{1/2}\|g\|_{L^2(\Delta_r(x_0))}\right).
 \end{aligned}
\end{equation}
Here \(C\) is independent of \(\varepsilon\), but may depend on \(\rho\), with no smallness condition on \(\varepsilon/r\).
\end{lemma}
The proof combines the argument for \eqref{eq:unified-local-pressure-control} in Appendix~\hyperref[app:local-energy-proof]{\ref*{app:quantitative-estimates}} with the \hyperref[app:mixed-energy-proof]{proof of Lemma~\ref*{lem:c1a-joint-caccioppoli}} in Appendix~\ref{app:interface-estimates}.
\end{samepage}

\begin{samepage}
\begin{lemma}[Local comparison]
\label{lem:boundary-effective-comparison}
Assume \textup{(G)} and \(\partial\Omega,\partial\omega\in C^{1,\kappa}\), where \(0<\kappa<1\).
Let \(x_0\in\partial\Omega\), \(0<r\le R_0/2\), and suppose that \((u_\varepsilon,p_\varepsilon)\)
satisfies the local weak formulation \eqref{pos:local-momentum} in \(U_{2r}(x_0)\), with \(g\in L^\infty(\Delta_{2r}(x_0);\R^d)\).
There exist constants \(c,C>0\) such that, if \(0<\varepsilon\le cr\), one can choose \(u_0\in H^1(U_r(x_0);\R^d)\) satisfying
\[
 \int_{U_r(x_0)}(\widehat A\D u_0):\D\varphi
 =\int_{\Delta_r(x_0)}g\cdot\varphi
\]
for every \(\varphi\in H^1(U_r(x_0);\R^d)\) with zero trace on \(\Omega\cap\partial B_r(x_0)\). It also satisfies
\begin{equation}
 \left(\dashint_{U_r(x_0)}|u_\varepsilon-u_0|^2\right)^{1/2}
 \le C\sqrt{\varepsilon/r}\left[
   \inf_{b\in\R^d}\left(\dashint_{U_{2r}(x_0)}|u_\varepsilon-b|^2\right)^{1/2}
   +r\|g\|_{L^\infty(\Delta_{2r}(x_0))}\right].
 \label{pos:local-comparison}
\end{equation}
Here \(c,C\) are independent of \(\varepsilon\) and depend only on \(d,\kappa,\lambda,\mu,\widetilde\mu,\Omega,\omega\).
\end{lemma}
\end{samepage}

\begin{proof}
\begingroup
\setlength{\abovedisplayskip}{3pt plus1pt minus1pt}
\setlength{\belowdisplayskip}{3pt plus1pt minus1pt}
\setlength{\abovedisplayshortskip}{0pt plus1pt}
\setlength{\belowdisplayshortskip}{3pt plus1pt minus1pt}
By the local weak formulation \eqref{pos:local-momentum} and Lemma~\ref{lem:boundary-local-energy}, one can choose \(s\in(5r/4,3r/2)\) and \(h\in L^2(\partial U_s(x_0);\R^d)\) such that \(h=g\) on \(\Delta_s(x_0)\) and, for every \(\varphi\in H^1(U_s(x_0);\R^d)\),
\begin{equation}
 \int_{U_s(x_0)\cap\Omega_\varepsilon}[\lambda\diver u_\varepsilon\,\diver\varphi+2\mu\D u_\varepsilon:\D\varphi]
 +\int_{U_s(x_0)\cap D_\varepsilon}[2\widetilde\mu\D u_\varepsilon:\D\varphi+p_\varepsilon\diver\varphi]
 =\int_{\partial U_s(x_0)}h\cdot\varphi.
 \label{pos:comparison-micro-weak}
\end{equation}
Moreover,
\begin{equation}
 \|h\|_{L^2(\partial U_s(x_0))}
 \le C\left[r^{-3/2}\inf_{b\in\R^d}\|u_\varepsilon-b\|_{L^2(U_{2r}(x_0))}
             +r^{(d-1)/2}\|g\|_{L^\infty(\Delta_{2r}(x_0))}\right].
 \label{pos:comparison-traction-bound}
\end{equation}
Take the test function \(\varphi=\ell\in\mathcal R\) in \eqref{pos:comparison-micro-weak} to obtain \(\int_{\partial U_s(x_0)}h\cdot\ell=0\).
Thus we may choose \(u_0\in H^1(U_s(x_0);\R^d)\) such that
\begin{equation}
 \begin{gathered}
 \int_{U_s(x_0)}(\widehat A\D u_0):\D\varphi
 =\int_{\partial U_s(x_0)}h\cdot\varphi,
 \qquad \varphi\in H^1(U_s(x_0);\R^d),\\
 \int_{\partial U_s(x_0)}(u_\varepsilon-u_0)\cdot\ell=0,
 \qquad \ell\in\mathcal R.
 \end{gathered}
 \label{pos:comparison-effective-weak}
\end{equation}
Set \(\Sigma_t(U_s(x_0))=\{x\in U_s(x_0):\operatorname{dist}(x,\partial U_s(x_0))<t\}\).
For \(u_0\) in \eqref{pos:comparison-effective-weak}, we have
\begin{equation}
 \|\D u_0\|_{L^2(\Sigma_t(U_s(x_0)))}
 +t\|\nabla^2u_0\|_{L^2(U_s(x_0)\setminus\Sigma_t(U_s(x_0)))}
 \le C\sqrt t\,\|h\|_{L^2(\partial U_s(x_0))},
 \qquad 0<t<r.
 \label{pos:comparison-layer-bound}
\end{equation}
Choose \(\ell\in\mathcal R\) such that \(u_0-\ell\perp\mathcal R\) in \(L^2(U_s(x_0))\).
For \(u_0-\ell\), the constant-coefficient \(L^2\) Neumann estimate and standard interior estimates give \eqref{pos:comparison-layer-bound}; see \cite[Eqs.~(2-22),~(2-24),~(2-29)]{Shen2017Boundary}.

Take \(L=2\sqrt d\), and let \(c\le1/(8L)\).
Choose \(\eta_\varepsilon\in C_c^\infty(U_s(x_0))\) satisfying
\begin{equation}\label{pos:comparison-cutoff-properties}
 \begin{gathered}
 0\le\eta_\varepsilon\le1,\qquad
 |\nabla\eta_\varepsilon|\le C\varepsilon^{-1},\\
 \eta_\varepsilon=0\quad\text{on }\Sigma_{2L\varepsilon}(U_s(x_0)),
 \qquad
 \eta_\varepsilon=1\quad\text{on }U_s(x_0)\setminus\Sigma_{4L\varepsilon}(U_s(x_0)).
 \end{gathered}
\end{equation}
By \eqref{pos:comparison-cutoff-properties}, \(Y_\varepsilon^k\cap\operatorname{supp}\eta_\varepsilon\ne\varnothing\) implies \(\overline{Y_\varepsilon^k}\subset U_s(x_0)\subset\Omega\) and hence \(k\in\Pi_\varepsilon\).
Define
\[
 \begin{aligned}
 \widetilde u_\varepsilon
 &=u_\varepsilon-u_0-\varepsilon\chi_j^\beta(x/\varepsilon)\eta_\varepsilon(\D u_0)_j^\beta
       &&\text{in }U_s(x_0),\\
 \widetilde p_\varepsilon
 &=p_\varepsilon-r_j^\beta(x/\varepsilon)\eta_\varepsilon(\D u_0)_j^\beta
       &&\text{in }D_\varepsilon\cap U_s(x_0).
 \end{aligned}
\]
Using \(\diver u_\varepsilon=0\) in \(D_\varepsilon\cap U_s(x_0)\), the cell divergence constraint \eqref{eq:unified-cell-momentum}, the bounds \eqref{eq:unified-bounded-correctors}, \eqref{pos:comparison-layer-bound}, and the cutoff properties \eqref{pos:comparison-cutoff-properties}, we obtain
\begin{equation}
 \|\diver\widetilde u_\varepsilon\|_{L^2(D_\varepsilon\cap U_s(x_0))}
 \le C\sqrt\varepsilon\,\|h\|_{L^2(\partial U_s(x_0))}.
 \label{pos:comparison-divergence}
\end{equation}
Subtracting \eqref{pos:comparison-effective-weak} from \eqref{pos:comparison-micro-weak} and using the definition of the flux discrepancy \eqref{eq:unified-flux-discrepancy} and its potential representation \eqref{eq:unified-flux-potential}, we obtain, for every \(\varphi\in H^1(U_s(x_0);\R^d)\),
\begin{equation}
 \begin{aligned}
 &\Bigl|\int_{U_s(x_0)\cap\Omega_\varepsilon}
 [\lambda\diver\widetilde u_\varepsilon\,\diver\varphi+2\mu\D\widetilde u_\varepsilon:\D\varphi]
 +\int_{U_s(x_0)\cap D_\varepsilon}
 [2\widetilde\mu\D\widetilde u_\varepsilon:\D\varphi+\widetilde p_\varepsilon\diver\varphi]\Bigr|\\
 &=\Bigl|\int_{U_s(x_0)}(1-\eta_\varepsilon)
 \Bigl[\widehat A\D u_0-\lambda\mathbf1_{\Omega_\varepsilon}
 (\diver u_0)\mathbb I
 -2(\mu\mathbf1_{\Omega_\varepsilon}
 +\widetilde\mu\mathbf1_{D_\varepsilon})\D u_0\Bigr]:\D\varphi\\
 &\qquad+\varepsilon\int_{U_s(x_0)}
 [\phi_{ij}^{\gamma\alpha\beta}
 -A_{ik}^{\alpha\gamma}(\chi_j^\beta)^k](x/\varepsilon)
 \partial_\gamma[\eta_\varepsilon(\D u_0)_j^\beta]
 \partial_\alpha\varphi^i\Bigr|\\
 &\qquad\le C\left[\|\D u_0\|_{L^2(\Sigma_{4L\varepsilon}(U_s(x_0)))}
 +\varepsilon\|\nabla(\eta_\varepsilon\D u_0)\|_{L^2(U_s(x_0))}\right]
 \|\nabla\varphi\|_{L^2(U_s(x_0))}\\
 &\qquad\le C\sqrt\varepsilon\,\|h\|_{L^2(\partial U_s(x_0))}
       \|\nabla\varphi\|_{L^2(U_s(x_0))}.
 \end{aligned}
 \label{pos:comparison-momentum}
\end{equation}
By \eqref{eq:unified-bounded-correctors} and \eqref{eq:unified-flux-potential}, \(\chi,\phi\) are bounded, and \eqref{pos:comparison-cutoff-properties} gives \(1-\eta_\varepsilon=0\) in \(U_s(x_0)\setminus\Sigma_{4L\varepsilon}(U_s(x_0))\). This yields the first inequality.

By \eqref{pos:comparison-effective-weak} and \eqref{pos:comparison-cutoff-properties}, \(\int_{\partial U_s(x_0)}\widetilde u_\varepsilon\cdot\ell=0\quad(\ell\in\mathcal R)\).
Apply the Babu\v{s}ka--Brezzi stability estimate to \((\widetilde u_\varepsilon,\widetilde p_\varepsilon)\), as in the proof of Proposition~\ref{thm:micro-full-pressure-infsup}. Using \eqref{pos:comparison-divergence}--\eqref{pos:comparison-momentum} and rescaling, we obtain
\begin{equation}
 \|\nabla \widetilde u_\varepsilon\|_{L^2(U_s(x_0))}
 +\frac{\|\widetilde u_\varepsilon\|_{L^2(U_s(x_0))}}r
 +\|\widetilde p_\varepsilon\|_{L^2(D_\varepsilon\cap U_s(x_0))}
 \le C\sqrt\varepsilon\,\|h\|_{L^2(\partial U_s(x_0))}.
 \label{pos:comparison-residual-bound}
\end{equation}
\begin{samepage}
Test \eqref{pos:comparison-effective-weak} with \(u_0-\ell\), where \(\ell\in\mathcal R\) was chosen so that \(u_0-\ell\perp\mathcal R\) in \(L^2(U_s(x_0))\). Korn's and trace inequalities, together with the coercivity from Proposition~\ref{thm:effective-tensor-properties}, give
\begin{equation}
 \|\D u_0\|_{L^2(U_s(x_0))}\le C\sqrt r\,\|h\|_{L^2(\partial U_s(x_0))}.
 \label{pos:comparison-effective-energy}
\end{equation}
\par\end{samepage}
\pagebreak[0]
\begin{samepage}
Using \eqref{pos:comparison-residual-bound} and \eqref{pos:comparison-effective-energy}, the corrector bound, \(\varepsilon\le cr\), and \eqref{pos:comparison-traction-bound}, we obtain
\begin{equation}
 \begin{aligned}
 \|u_\varepsilon-u_0\|_{L^2(U_s(x_0))}
 &\le\|\widetilde u_\varepsilon\|_{L^2(U_s(x_0))}
       +C\varepsilon\|\D u_0\|_{L^2(U_s(x_0))}\\
 &\le Cr\sqrt\varepsilon\,\|h\|_{L^2(\partial U_s(x_0))}\\
 &\le C\sqrt{\varepsilon/r}
       \left[\inf_{b\in\R^d}\|u_\varepsilon-b\|_{L^2(U_{2r}(x_0))}
             +r^{1+d/2}\|g\|_{L^\infty(\Delta_{2r}(x_0))}\right].
 \end{aligned}
 \label{pos:comparison-displacement-bound}
\end{equation}
\par\end{samepage}
Take any \(\varphi\in H^1(U_r(x_0);\R^d)\) with zero trace on \(\Omega\cap\partial B_r(x_0)\).
Extend it by zero to \(U_s(x_0)\) and test \eqref{pos:comparison-effective-weak}.
Since \(h=g\) on \(\Delta_s(x_0)\), this gives the local weak formulation in the lemma.

Restrict \eqref{pos:comparison-displacement-bound} to \(U_r(x_0)\) and normalize using
\(|U_r(x_0)|\simeq|U_{2r}(x_0)|\simeq r^d\) to obtain \eqref{pos:local-comparison}.
\endgroup
\end{proof}

\begin{samepage}
\begin{proposition}[Uniform boundary Lipschitz estimate]
\label{thm:local-boundary-lipschitz}
Assume \textup{(G)} and \(\partial\Omega,\partial\omega\in C^{1,\kappa}\), where \(0<\tau<\kappa<1\).
Let \(0<\varepsilon<1\), \(x_0\in\partial\Omega\), \(0<R\le R_0\), and suppose that \((u_\varepsilon,p_\varepsilon)\) satisfies the local weak formulation \eqref{pos:local-momentum} in \(U_R(x_0)\), with \(g\in C^{0,\tau}(\Delta_R(x_0);\R^d)\). Then
\begin{equation}\label{pos:local-boundary-result}
 \|\nabla u_\varepsilon\|_{L^\infty(U_{R/64}(x_0))} \le C\biggl[\frac1R\inf_{b\in\R^d} \left(\dashint_{U_R(x_0)}|u_\varepsilon-b|^2\right)^{1/2} +\|g\|_{L^\infty(\Delta_R(x_0))} +R^\tau[g]_{C^{0,\tau}(\Delta_R(x_0))}\biggr].
\end{equation}
Here \(C\) is independent of \(\varepsilon,R,x_0\) and depends only on \(d,\kappa,\tau,\lambda,\mu,\widetilde\mu,\Omega,\omega\).
\end{proposition}
\end{samepage}

\begin{proof}
\begingroup
\setlength{\abovedisplayskip}{2pt plus1pt minus1pt}
\setlength{\belowdisplayskip}{2pt plus1pt minus1pt}
\setlength{\abovedisplayshortskip}{0pt plus1pt}
\setlength{\belowdisplayshortskip}{2pt plus1pt minus1pt}
For \(y\in\partial\Omega\) and \(U_r(y)\subset U_R(x_0)\), write
\begin{equation}\label{pos:boundary-data-functional}
 \mathcal A_r(y)=\frac1r\inf_{b\in\R^d}
 \left(\dashint_{U_r(y)}|u_\varepsilon-b|^2\right)^{1/2}
 +\|g\|_{L^\infty(\Delta_r(y))}
 +r^\tau[g]_{C^{0,\tau}(\Delta_r(y))}.
\end{equation}
\begin{samepage}
For \(0<r\le R\) and affine functions \(\ell(x)=b+M(x-x_0)\), define
\begin{equation}\label{pos:fitting}
 \begin{aligned}
 \mathcal J(r)=\inf_{\substack{b\in\R^d\\M\in\R^{d\times d}}}\biggl\{
 &\frac1r\left(\dashint_{U_r(x_0)}|u_\varepsilon-\ell|^2\right)^{1/2}
 +\|g-(\widehat A M)n\|_{L^\infty(\Delta_r(x_0))}\\
 &+r^\tau[g-(\widehat A M)n]_{C^{0,\tau}(\Delta_r(x_0))}\biggr\}.
 \end{aligned}
\end{equation}
\end{samepage}
Choose an affine function \(\ell_r\) attaining the infimum in \eqref{pos:fitting}. For every affine function \(\ell\),
\begin{equation}\label{pos:basic-fitting-bounds}
 \begin{gathered}
 |\nabla\ell|\le\frac Cr\inf_{b\in\R^d}
 \left(\dashint_{U_r(x_0)}|\ell-b|^2\right)^{1/2},\\
 \mathcal A_r(x_0)\le C[\mathcal J(r)+|\nabla\ell_r|],\\
 \mathcal J(R)+|\nabla\ell_R|\le C\mathcal A_R(x_0).
 \end{gathered}
\end{equation}

For \(0<r\le R\) and sufficiently small \(\varepsilon/r\), apply Lemma~\ref{lem:boundary-effective-comparison} at radius \(r/2\) to obtain \(u_0\).
Apply the constant-coefficient Neumann boundary \(C^{1,\tau}\) estimate (see \cite[proof of Lemma~9.3]{Shen2017Boundary}) to \(u_0-\ell_r\), and use the first-order Taylor polynomial of \(u_0\) at \(x_0\) in \eqref{pos:fitting}.
The comparison bound and the second inequality in \eqref{pos:basic-fitting-bounds}, with \(0<\theta<1/8\) sufficiently small, give
\[
 \mathcal J(\theta r)
 \le\tfrac14\mathcal J(r)
 +C\sqrt{\varepsilon/r}\,[\mathcal J(r)+|\nabla\ell_r|].
\]
Fix \(L_0\ge1\) large enough. The preceding bound and the first bound in \eqref{pos:basic-fitting-bounds}, applied to \(\ell_{\theta r}-\ell_r\), give
\begin{equation}\label{pos:two-recursions}
 \begin{aligned}
 \mathcal J(\theta r)&\le\tfrac12\mathcal J(r)
       +C\sqrt{\varepsilon/r}\,|\nabla\ell_r|,\\
 |\nabla\ell_{\theta r}-\nabla\ell_r|
       &\le C[\mathcal J(r)+\mathcal J(\theta r)],
 \qquad L_0\varepsilon\le r\le R.
 \end{aligned}
\end{equation}
Here \(\theta,L_0,C\) depend only on the parameters listed in the proposition.

The boundary energy estimate \eqref{pos:boundary-energy-bound} and volume comparison give
\begin{equation}\label{pos:boundary-coarse-gradient}
 \left(\dashint_{U_t(x_0)}|\nabla u_\varepsilon|^2\right)^{1/2}
 \le C\mathcal A_R(x_0),
 \qquad R/4\le t\le R/2.
\end{equation}
It remains to consider \(L_0\varepsilon\le t<R/4\), so we may assume \(R>4L_0\varepsilon\).

\par\pagebreak[0]
Set \(r_j=\theta^jR\) and choose \(k\ge0\) so that \(r_k\ge L_0\varepsilon>r_{k+1}\). Then
\[
 \sum_{j=0}^k\sqrt{\frac{\varepsilon}{r_j}}
 \le\frac{1}{(1-\sqrt\theta)\sqrt{L_0}}.
\]
Using this bound, sum the two inequalities in \eqref{pos:two-recursions} and take \(L_0\) sufficiently large to absorb the slope term. Together with \eqref{pos:basic-fitting-bounds}, this gives
\begin{equation}\label{pos:finite-scale-sums}
 \begin{aligned}
 \sum_{j=0}^k\mathcal J(r_j)+\max_{0\le j\le k}|\nabla\ell_{r_j}|
 &\le C[\mathcal J(R)+|\nabla\ell_R|]\le C\mathcal A_R(x_0).
 \end{aligned}
\end{equation}
\par\pagebreak[0]
\begin{samepage}
For \(L_0\varepsilon\le t<R/4\), choose \(2t\le r_j\le2t/\theta\). In \(U_{r_j}(x_0)\), apply the boundary energy estimate \eqref{pos:boundary-energy-bound} to \(u_\varepsilon-\ell_{r_j}(x_0)\). Combining the second inequality in \eqref{pos:basic-fitting-bounds}, \eqref{pos:finite-scale-sums}, and volume comparison gives
\[
 \left(\dashint_{U_t(x_0)}|\nabla u_\varepsilon|^2\right)^{1/2}
 \le C[\mathcal J(r_j)+|\nabla\ell_{r_j}|]
 \le C\mathcal A_R(x_0).
\]
\par\end{samepage}
\par\pagebreak[0]
\begin{samepage}
Combining this with \eqref{pos:boundary-coarse-gradient} gives
\begin{equation}\label{pos:large-scale-gradient}
 \left(\dashint_{U_t(x_0)}|\nabla u_\varepsilon|^2\right)^{1/2}
 \le C\mathcal A_R(x_0),
 \qquad \min\{L_0\varepsilon,R/4\}\le t\le R/2.
\end{equation}
\par\end{samepage}

For \(x\in U_{R/64}(x_0)\), let \(y\in\partial\Omega\) be a nearest boundary point and set \(h=\min\{L_0\varepsilon,R/8\}\).
Then \(|x-y|<R/64\), \(|y-x_0|<R/32\), and \(\mathcal A_{R/2}(y)\le C\mathcal A_R(x_0)\).
Applying \eqref{pos:large-scale-gradient} with center \(y\) and outer radius \(R/2\) gives
\begin{equation}\label{pos:centered-boundary-gradient}
 \left(\dashint_{U_t(y)}|\nabla u_\varepsilon|^2\right)^{1/2}
 \le C\mathcal A_R(x_0),\qquad h\le t\le R/4.
\end{equation}
By \textup{(G)} and \(h\le L_0\varepsilon\), choose \(0<a<1/16\), depending only on the parameters listed in the proposition, so that \(U_{4ah}(y)\) contains only the Lam\'e phase.

If \(|x-y|<ah\), apply the constant-coefficient Neumann boundary Lipschitz estimate \cite[Corollary~9.7]{Shen2017Boundary} on \(U_{4ah}(y)\subset U_h(y)\).
If \(|x-y|\ge ah\), set \(s=\max\{2|x-y|,h\}\). Then
\[
 B_{|x-y|/2}(x)\subset U_s(y),\qquad
 h\le s\le R/8,\qquad \frac{s}{|x-y|}\le\frac1a,
\]
Apply Proposition~\ref{thm:main-interior-lipschitz} with center \(x\) and radius \(|x-y|/4\).
In the two cases, use \eqref{pos:centered-boundary-gradient} with \(t=h\) and \(t=s\), respectively.
Together with \eqref{pos:boundary-data-functional}, Poincar\'e's inequality, and volume comparison, this gives
\[
 |\nabla u_\varepsilon(x)|\le C\mathcal A_R(x_0)
 \quad\text{for a.e. }x\in U_{R/64}(x_0).
\]
Taking the essential supremum proves \eqref{pos:local-boundary-result}.
\endgroup
\end{proof}

\begin{samepage}
\begin{proposition}[Pressure estimate in terms of the gradient]
\label{pos:gradient-to-pressure}
Assume \textup{(G)}. Let \(0<\varepsilon<1\), and suppose that \((u_\varepsilon,p_\varepsilon)\)
satisfies, on every interior ball \(B\Subset\Omega\), the local weak formulation \eqref{eq:il-homogeneous-system} with inclusion region \(D_\varepsilon\).
If \(\nabla u_\varepsilon\in L^\infty(\Omega)\), then
\begin{equation}\label{pos:global-pressure-from-gradient}
 \|p_\varepsilon\|_{L^\infty(D_\varepsilon)}
 \le C\|\nabla u_\varepsilon\|_{L^\infty(\Omega)}.
\end{equation}
Here \(C\) is independent of \(\varepsilon\) and depends only on \(d,\lambda,\mu,\widetilde\mu,\omega\).
\end{proposition}
\end{samepage}

\begin{proof}
Fix \(k\in\Pi_\varepsilon\). By a partition of unity and density, the momentum identity in
\eqref{eq:il-homogeneous-system} also holds for test functions in
\(H_0^1(Y_\varepsilon^k;\R^d)\).

Fix \(x\in\omega_\varepsilon^k\) and set
\(\delta=\operatorname{dist}(x,\partial\omega_\varepsilon^k)\).
Since \(\omega\Subset Y\) and \(\partial\omega\) is Lipschitz, there is \(c>0\),
depending only on the fixed geometry, such that when \(0<\delta<c\varepsilon\),
a nearest interface point \(z\) satisfies
\[
 B_\delta(z)\subset B_{2\delta}(x)
 \subset B_{3\delta}(z)\Subset Y_\varepsilon^k
\]
and \(|B_\delta(z)\cap\Omega_\varepsilon|\ge C^{-1}\delta^d\).
Together with the scaling of the cell, this gives
\begin{equation}\label{pos:pressure-volume-bounds}
 \frac{|Y_\varepsilon^k\setminus\overline{\omega_\varepsilon^k}|}
 {|Y_\varepsilon^k|}
 =\frac{|Y\setminus\overline\omega|}{|Y|}>0,
 \qquad |B_{2\delta}(x)\cap\Omega_\varepsilon|
 \ge C^{-1}|B_{2\delta}(x)|,
 \quad 0<\delta<c\varepsilon.
\end{equation}
Using \eqref{pos:pressure-volume-bounds} and the argument in Appendix~\hyperref[app:local-energy-proof]{\ref*{app:quantitative-estimates}} for \eqref{eq:unified-local-pressure-control}, we obtain
\begin{equation}\label{pos:local-pressure-bounds}
 \begin{aligned}
 \|p_\varepsilon\|_{L^2(\omega_\varepsilon^k)}
 &\le C\|\nabla u_\varepsilon\|_{L^2(Y_\varepsilon^k)},\\
 \|p_\varepsilon\|_{L^2(B_{2\delta}(x)\cap D_\varepsilon)}
 &\le C\|\nabla u_\varepsilon\|_{L^2(B_{2\delta}(x))},
 \qquad 0<\delta<c\varepsilon.
 \end{aligned}
\end{equation}

In the inclusion, \(\Delta p_\varepsilon=0\). Thus, the harmonic mean-value formula and the Cauchy--Schwarz inequality give
\[
 |p_\varepsilon(x)|\le Cr^{-d/2}\|p_\varepsilon\|_{L^2(B_r(x))},
 \qquad B_r(x)\Subset\omega_\varepsilon^k.
\]
If \(\delta\ge c\varepsilon\), take \(r=c\varepsilon/2\) and use the cell estimate in
\eqref{pos:local-pressure-bounds}; if \(0<\delta<c\varepsilon\),
take \(r=\delta/2\) and use the ball estimate there. This yields
\begin{equation}\label{pos:pointwise-pressure-recovery}
 |p_\varepsilon(x)|
 \le C
 \begin{cases}
 \varepsilon^{-d/2}
 \|\nabla u_\varepsilon\|_{L^2(Y_\varepsilon^k)},
 &\delta\ge c\varepsilon,\\[3pt]
 \delta^{-d/2}
 \|\nabla u_\varepsilon\|_{L^2(B_{2\delta}(x))},
 &0<\delta<c\varepsilon
 \end{cases}
 \le C\|\nabla u_\varepsilon\|_{L^\infty(Y_\varepsilon^k)}.
\end{equation}
The last inequality uses \(|Y_\varepsilon^k|=\varepsilon^d|Y|\) and \(|B_{2\delta}(x)|\le C\delta^d\).
Taking the essential supremum over all inclusions yields \eqref{pos:global-pressure-from-gradient}.
\end{proof}

\begin{proof}[Proof of Theorem~\ref{thm:positive-boundary-pressure}]
The interior estimate (Proposition~\ref{thm:main-interior-lipschitz}), the boundary estimate (Proposition~\ref{thm:local-boundary-lipschitz}), and a finite covering yield
\[
 \|\nabla u_\varepsilon\|_{L^\infty(\Omega)}
 \le C\Bigl(\|u_\varepsilon\|_{L^2(\Omega)}
 +\|g\|_{C^{0,\tau}(\partial\Omega)}\Bigr)
 \le C\|g\|_{C^{0,\tau}(\partial\Omega)},
\]
where the last inequality uses the stability estimate in Proposition~\ref{thm:micro-full-pressure-infsup}. Proposition~\ref{pos:gradient-to-pressure} gives \eqref{pos:global-result}.
\end{proof}

\appendix
\section{Proof of Proposition~\ref{thm:micro-full-pressure-infsup}}
\label{app:mixed-stability-proof}

\begin{proof}[]
We apply the Babu\v{s}ka--Brezzi theorem \cite[Chapter~4]{BoffiBrezziFortin2013}
on \(V\times L^2(D_\varepsilon)\).
Both forms are continuous uniformly in \(\varepsilon\); material positivity
and Korn's inequality give coercivity of \(\mathfrak a_\varepsilon\) on \(V\)
uniformly in \(\varepsilon\).
To obtain an inf--sup constant independent of \(\varepsilon\), we construct
a test field on the fixed domain \(\Omega\) whose divergence in
\(D_\varepsilon\) equals \(q\).

For any \(q\in L^2(D_\varepsilon)\), set
\(f=q\mathbf1_{D_\varepsilon}\in L^2(\Omega)\).
The Bogovski\u{\i} operator
\cite[Theorem~4.1]{AcostaDuranMuschietti2006} gives
\(z\in H^1_0(\Omega;\R^d)\) with
\[
 \diver z=f-\fint_\Omega f,
 \qquad
 \|z\|_{H^1(\Omega)}\le C\|q\|_{L^2(D_\varepsilon)}.
\]
The traction boundary condition permits restoring the mean by an affine field.
Choose \(r\in\mathcal R\) so that
\[
 v=z+\frac{x}{d}\fint_\Omega f-r\in V.
\]
Since \(\diver r=0\) and \(\D r=0\), we have \(\diver v=f\).
By the Cauchy--Schwarz inequality and Korn's inequality on \(V\),
\[
 \|v\|_V
 \le C\|\D v\|_{L^2(\Omega)}
 \le C\|q\|_{L^2(D_\varepsilon)},
\]
where \(C=C(d,\Omega)\). Together with
\(b_\varepsilon(v,q)=\|q\|_{L^2(D_\varepsilon)}^2\), this proves
\eqref{eq:micro-full-pressure-infsup}.

Finally, the trace theorem controls the load \(\langle g,v\rangle_{\partial\Omega}\).
Together with \eqref{eq:micro-full-pressure-infsup}, this proves
well-posedness and the energy bound.
\end{proof}

\section{Auxiliary results and proofs for local interface estimates}
\label{app:interface-estimates}\label{app:single-phase-estimates}

We use the notation of \eqref{eq:main-interface-system} and the ranges of \(k,\kappa\) in Theorem~\ref{thm:main-interface-schauder}.
The interior Schauder estimates for \(\nabla u\) in the Lam\'e system and \((\nabla u,p)\) in the Stokes system follow from \cite[Theorems~5.19--5.20]{GiaquintaMartinazzi2012} and \cite[Theorem~4.4]{NguyenStara2007}, respectively, with lower-order quantities \(\|\nabla u\|_{L^2(B_1)}\) and \(\|\nabla u\|_{L^2(B_1)}+\|p\|_{L^2(B_1)}\).

The standard Stokes pressure is \(-p-\widetilde\mu f\); the mean of \(p\) is controlled by its \(L^2\) norm.
Higher-order estimates follow by differentiating the constant-coefficient equations. All constants depend only on \(d,k,\kappa\) and the corresponding material parameters.

\Needspace{6\baselineskip}
\begin{lemma}[Local Dirichlet boundary estimates]
\label{lem:single-phase-boundary-schauder}
Let \(0<\kappa<1\), and suppose that the interface \(B_1\cap\partial\omega\) is a \(C^{1,\kappa}\) graph, with
\(G_+\in C^\kappa(\overline{B_1\setminus\overline\omega})\) and
\(G_-,f\in C^\kappa(\overline{B_1\cap\omega})\).
If \((u,p)\) is a weak solution of system \eqref{eq:main-interface-system} and its interface trace satisfies
\(u|_{B_1\cap\partial\omega}\in C^{1,\kappa}(B_1\cap\partial\omega)\), then
\[
\begin{aligned}
 &\|u\|_{C^{1,\kappa}(\overline{B_{1/2}\setminus\overline\omega})}
 +\|u\|_{C^{1,\kappa}(\overline{B_{1/2}\cap\omega})}
 +\|p\|_{C^\kappa(\overline{B_{1/2}\cap\omega})}\\
 &\quad\le C\Bigl\{\|u\|_{H^1(B_1)}+\|p\|_{L^2(B_1\cap\omega)}
 +\|G_+\|_{C^\kappa(\overline{B_1\setminus\overline\omega})}\\
 &\hspace{4em}+\|G_-\|_{C^\kappa(\overline{B_1\cap\omega})}
 +\|f\|_{C^\kappa(\overline{B_1\cap\omega})}
 +\|u\|_{C^{1,\kappa}(B_1\cap\partial\omega)}\Bigr\}.
\end{aligned}
\]
The constant \(C\) depends only on \(d,\kappa,\lambda,\mu,\widetilde\mu\) and the \(C^{1,\kappa}\) geometry of the interface.
\end{lemma}
After subtracting a \(C^{1,\kappa}\) extension of the interface trace, apply and add the local boundary estimates in \cite[Theorem~5.21 and its proof]{GiaquintaMartinazzi2012} and \cite[proof of Theorem~1.4(b)]{ChoiDong2019Domains} to obtain the stated estimate; use the pressure conversion described at the beginning of this appendix. The mean of \(p\) is controlled by \(\|p\|_{L^2(B_1\cap\omega)}\).

\begin{proof}[Proof of Lemma~\ref{lem:c1a-joint-caccioppoli}]\phantomsection\label{app:mixed-energy-proof}
For \(r/2\le s<t<r\), set \(V=u-b\) and choose a standard cutoff \(\eta\) for \(B_s\subset B_t\). By \eqref{eq:interface-structure} and Lemma~\ref{lem:c1a-rough-piola-lift}, choose \(w\in H^1_0(B_t;\R^d)\) such that \(\diver(aw)=\diver(a\eta^2V)=\eta^2f+2\eta\nabla\eta\cdot aV\) in \(B_t^-\), with the gradient bound supplied by that lemma.

Test the weak momentum identity of system \eqref{eq:c1a-variable-system} with \(\eta^2V-w\). Coefficient symmetry, coercivity \eqref{eq:schauder-chart-coercivity}, Young's inequality, and the standard iteration lemma \cite[Lemma~4.3]{HanLin2011} give the gradient estimate.
The full-pressure inf--sup estimate of Remark~\ref{rem:c1a-full-pressure-infsup} and the weak momentum identity then control \(p\), giving \eqref{eq:c1a-joint-caccioppoli}.
\end{proof}

\section{Auxiliary proofs for quantitative homogenization}
\label{app:quantitative-estimates}

\begin{proof}[Proof of Lemma~\ref{lem:unified-local-energy}]\phantomsection\label{app:local-energy-proof}
Extend \(p\) by a constant on \(B_t\setminus D_\varepsilon\) to have zero integral over the ball.
The volume-fraction condition and the divergence right inverse on a ball
\cite[Theorem~3.2]{Duran2012} give \(\varphi\in H_0^1(B_t;\R^d)\) such that
\[
 \diver\varphi=p\quad\text{in }D_\varepsilon\cap B_t,\qquad
 \|\nabla\varphi\|_{L^2(B_t)}
 \le C\|p\|_{L^2(D_\varepsilon\cap B_t)}.
\]
Extend it by zero to \(B_R\) and test the first equation in \eqref{eq:unified-local-mixed-problem}
to obtain \eqref{eq:unified-local-pressure-control}.

Take \(R/2\le s<t\le R\), \(b\in\R^d\), and, for \(B_s\subset B_t\), a standard cutoff \(\eta\) satisfying \(|\nabla\eta|\le C/(t-s)\).
Test the first equation in \eqref{eq:unified-local-mixed-problem} with \(\eta^2(u-b)\).
Since, in \(D_\varepsilon\cap B_t\), \(\diver u=f\), estimate \eqref{eq:unified-local-pressure-control}, Young's inequality, and Korn's inequality for \(\eta(u-b)\) give
\[
\begin{aligned}
 \int_{B_s}|\nabla u|^2
 &\le \frac12\int_{B_t}|\nabla u|^2
 +\frac{C\|u-b\|_{L^2(B_R)}^2}{(t-s)^2}\\
 &\quad+C\left[
 \|G_+\|_{L^2(B_R\cap\Omega_\varepsilon)}^2
 +\|G_-\|_{L^2(B_R\cap D_\varepsilon)}^2
 +\|f\|_{L^2(D_\varepsilon\cap B_R)}^2\right].
\end{aligned}
\]
The iteration lemma \cite[Lemma~4.3]{HanLin2011} yields the gradient estimate.
Taking \(t=r\) in \eqref{eq:unified-local-pressure-control} and then the infimum over \(b\in\R^d\) gives \eqref{eq:unified-inhomogeneous-caccioppoli}.
Apply this estimate to \(u-\ell\), \(\ell\in\mathcal R\), and take the infimum to obtain the rigid-displacement version.
\end{proof}

\begin{proof}[Proof of Lemma~\ref{lem:unified-cell-higher-integrability}]\phantomsection\label{app:cell-higher-integrability-proof}
Extend \(\chi_j^\beta,r_j^\beta\) periodically and set \(U=p_j^\beta+\chi_j^\beta\). By \eqref{eq:unified-cell-momentum}, \((U,r_j^\beta)\) solves the homogeneous cell system.
For sufficiently small \(\rho\), apply the purely elastic case of \eqref{eq:unified-inhomogeneous-caccioppoli} or the homogeneous Stokes estimate \cite[Theorem~3.1]{GuShen2015} if \(B_{2\rho}(x)\) lies in one phase.
Otherwise, choose an interface point \(z\in B_{2\rho}(x)\) and apply \eqref{eq:unified-inhomogeneous-caccioppoli} on \(B_{6\rho}(z)\), with inner radius \(3\rho\) and \(G_\pm=f=0\); the required volume-fraction bound follows from the Lipschitz geometry.
Sobolev--Poincar\'e then gives, in either case,
\begin{equation}\label{eq:unified-cell-reverse-holder}
 \left(\fint_{B_\rho(x)}|\nabla U|^2\right)^{1/2}
 \le C\left(\fint_{B_{8\rho}(x)}
 |\nabla U|^{2d/(d+1)}\right)^{(d+1)/(2d)}.
\end{equation}
Gehring's lemma \cite[Theorem~6.38 and its proof]{GiaquintaMartinazzi2012}, a finite covering, the periodicity of \(\nabla U\), and \eqref{eq:cell-pair-energy-bound} give \(\|\nabla U\|_{L^s(Y)}\le C\) for some \(s>2\).
Since \(p_j^\beta\) is a fixed affine displacement and \(\chi_j^\beta\) has zero mean, Poincar\'e's inequality gives \(\|\chi_j^\beta\|_{W^{1,s}(Y)}\le C\).

Let \(s'=s/(s-1)\). Extend \(h\in C_c^\infty(\omega)\) by a constant on \(Y\setminus\overline\omega\) to have zero integral over the cell. The \(L^2,L^{s'}\) estimates for the periodic Poisson equation give \(\varphi\in H^1_{\#,0}(Y;\R^d)\) satisfying
\[
 \diver\varphi=h\quad\text{in }\omega,\qquad
 \|\nabla\varphi\|_{L^{s'}(Y)}\le C\|h\|_{L^{s'}(\omega)}.
\]
Testing the weak cell equation gives
\[
 \left|\int_\omega r_j^\beta h\right|
 =\left|\int_Y A(y)\D U:\D\varphi\right|
 \le C\|h\|_{L^{s'}(\omega)}.
\]
Duality yields \(\|r_j^\beta\|_{L^s(\omega)}\le C\).
\end{proof}

\section*{Acknowledgments}
The author thanks Wenjia Jing for suggesting the research topic and Yi Zhu for helpful comments on the exposition.

The author used OpenAI Codex to assist with manuscript editing, reference checking, and review of mathematical arguments. The author assumes responsibility for all content.


\end{document}